\documentclass[reqno, 11pt]{amsart}
\usepackage{url}
\usepackage[colorlinks]{hyperref}
\usepackage{helvet}

\usepackage{amsxtra}
\usepackage{amsmath}
\usepackage{amsthm}
\usepackage{amssymb}
\usepackage{amsbsy}
\usepackage{mathrsfs}

\newtheorem{theorem}{Theorem}[section]
\newtheorem{lemma}[theorem]{Lemma}

\newtheorem{proposition}[theorem]{Proposition}
\newtheorem{corollary}[theorem]{Corollary}

\theoremstyle{definition}

\newtheorem{remark}[theorem]{Remark}

\numberwithin{equation}{section}

\usepackage{bbm}
\usepackage{url}
\usepackage{euscript}
\usepackage{pb-diagram}
\usepackage{lamsarrow}
\usepackage{pb-lams}
\usepackage{amsmath}
\usepackage{amsthm}

\usepackage{amsxtra}
\usepackage{amssymb}
\usepackage{pifont}
\usepackage{amsbsy}

\usepackage{graphicx}
\usepackage{epstopdf}

\newskip\aline \newskip\halfaline
\def\skipaline{\vskip\aline}

\def\qedbox{$\rlap{$\sqcap$}\sqcup$}
\def\qed{\nobreak\hfill\penalty250 \hbox{}\nobreak\hfill\qedbox\skipaline}

\def\proofend{\eqno{\mbox{\qedbox}}}

\newcommand{\one}{\mathbbm{1}}

\newcommand\bC{{\mathbb C}}

\newcommand\bR{{\mathbb R}}

\newcommand{\bT}{{\mathbb T}}

\newcommand\bZ{{\mathbb Z}}

\DeclareMathOperator{\tr}{{\rm tr}}

\DeclareMathOperator{\dist}{dist}

\DeclareMathOperator{\diag}{Diag}

\DeclareMathOperator{\spa}{span}

\DeclareMathOperator{\Sym}{\mathbf{Sym}}

\DeclareMathOperator{\Hess}{Hess}

\DeclareMathOperator{\var}{\boldsymbol{var}}
\DeclareMathOperator{\SO}{SO}

\DeclareMathOperator{\GOE}{GOE}

\newcommand{\be}{{\boldsymbol{e}}}
\newcommand{\bsf}{\boldsymbol{f}}

\newcommand{\ii}{\boldsymbol{i}}

\newcommand{\bp}{{\boldsymbol{p}}}
\newcommand{\bq}{{\boldsymbol{q}}}

\newcommand{\bu}{{\mathbf{u}}}
\newcommand{\bv}{{\boldsymbol{v}}}

\newcommand{\bx}{{\boldsymbol{x}}}

\newcommand{\bsD}{\boldsymbol{D}}
\newcommand{\bsE}{\boldsymbol{E}}

\newcommand{\bsH}{\boldsymbol{H}}
\newcommand{\bsI}{\boldsymbol{I}}
\newcommand{\bsJ}{\boldsymbol{J}}

\newcommand{\bsN}{\boldsymbol{N}}

\newcommand{\bsS}{\mathbf{S}}
\newcommand{\bsT}{\boldsymbol{T}}
\newcommand{\bsU}{{\mathbf{U}} }
\newcommand{\bsV}{\boldsymbol{V}}
\newcommand{\bsW}{\boldsymbol{W}}

\newcommand{\bgamma}{\boldsymbol{\gamma}}
\newcommand{\bGamma}{\boldsymbol{\Gamma}}

\newcommand{\bpi}{\boldsymbol{\pi}}

\newcommand{\bsi}{\boldsymbol{\sigma}}

\newcommand{\si}{{\sigma}}

\newcommand{\ve}{{\varepsilon}}

\newcommand{\vfi}{{\varphi}}

\newcommand{\eA}{\EuScript{A}}

\newcommand{\eC}{\EuScript{C}}

\newcommand{\eE}{\EuScript{E}}

\newcommand{\eG}{\EuScript{G}}

\newcommand{\eH}{\EuScript H}

\newcommand{\eN}{\EuScript{N}}
\newcommand{\eO}{\EuScript{O}}

\newcommand{\eQ}{\EuScript{Q}}
\newcommand{\eR}{\EuScript{R}}
\newcommand{\eS}{\EuScript{S}}
\newcommand{\eT}{\EuScript{T}}

\newcommand{\teE}{\widetilde{\mathscr{E}}}

\newcommand{\hra}{\hookrightarrow}
\newcommand{\Lra}{{\longrightarrow}}

\newcommand{\lan}{\langle}
\newcommand{\ran}{\rangle}

\def\inpr{\mathbin{\hbox to 6pt{\vrule height0.4pt width5pt depth0pt \kern-.4pt \vrule height6pt width0.4pt depth0pt\hss}}}

\newcommand{\pa}{\partial}

\newcommand{\dual}{\spcheck{}}

\newcommand{\reh}{\widetilde{\boldsymbol{H}}^\varepsilon}
\newcommand{\coh}{\widehat{\boldsymbol{H}}^\varepsilon}
\newcommand{\bbom}{\bar{\boldsymbol{\Xi}}}
\newcommand{\tom}{\widetilde{\Omega}}

\newcommand{\tsi}{\tilde{\sigma}}

\begin{document}

\title[Multidimensional random Fourier series]{Critical points  of  multidimensional  random Fourier series: variance estimates} 

%\date{Started  July 27, 2012. Completed  on  October 8, 2013. Last modified on {\today}. }

\author{Liviu I. Nicolaescu}

\address{Department of Mathematics, University of Notre Dame, Notre Dame, IN 46556-4618.}
\email{nicolaescu.1@nd.edu}
\urladdr{\url{http://www.nd.edu/~lnicolae/}}
\subjclass{28A33, 60D99, 60G15, 60G60}
\keywords{random Fourier series, critical points, energy landscape, Kac-Rice formula, asymptotics, Gaussian random matrices}

\begin{abstract}  We investigate the number of critical points  of a Gaussian  random smooth function $u^\ve$ on the $m$-torus  $T^m:=\mathbb{R}^m/\mathbb{Z}^m$. The randomness  is  specified  by a  fixed  nonnegative Schwartz  function $w$ on the real axis and  a small parameter $\varepsilon$ so that,  as $\varepsilon \to 0$, the random function $u^\ve$ becomes highly oscillatory  and converges in a special fashion to  the Gaussian white noise.   Let $N^\varepsilon$ denote the number of critical points of $u^\ve$. We  describe explicitly   in terms of $w$   two  constants $C, C'$ such that  as $\varepsilon$ goes to the  zero, the expectation  of the random variable $\ve^{-m}N^\ve$ converges to $C$, while its variance  is extremely small and   behaves like $C'\ve^{m}$.

\end{abstract}

\maketitle

\tableofcontents

\section{Introduction}
\setcounter{equation}{0}

\subsection{The setup} Consider the $m$-dimensional torus $\bT^m:=\bR^m/\bZ^m$  with  angular coordinates $\theta_1,\dotsc, \theta_m\in\bR/\bZ$  equipped with  the flat metric
\[
g:=\sum_{j=1}^m (d\theta_j)^2,\;\;{\rm vol}_g(\bT^m)=1.
\]
The eigenvalues of the corresponding  Laplacian $\Delta=-\sum_{k=1}^m \pa^2_{\theta_k}$ are
\[
(2\pi)^2|\vec{k}|^2,\;\;\vec{k}=(k_1,\dotsc, k_m)\in\bZ^m.
\]
For $\vec{\theta}=(\theta_1,\dotsc, \theta_m) \in\bR^m$ and $\vec{k}\in\bZ^m$ we set
\[
 \lan\vec{k},\vec{\theta}\ran =\sum_jk_j\theta_j. 
 \]
Denote by $\prec$ the lexicographic order on $\bR^m$.  An  orthonormal basis of $L^2(\bT^m)$ is given by the  functions $(\Psi_{\vec{k}})_{\vec{k}\in\bZ^m}$, where
\[
\Psi_{\vec{k}}(\vec{\theta}) =\begin{cases}1, &\vec{k}=0\\
\sqrt{2} \sin 2\pi\lan \vec{k},\vec{\theta}\ran, &\vec{k}\succ\vec{0}, \\
\sqrt{2}\cos 2\pi\lan\vec{k},\vec{\theta}\ran, &\vec{k}\prec\vec{0}. 
\end{cases}
\]
Fix a  nonnegative, even  Schwartz function $w\in \eS(\bR)$, set $w_\ve(t)=w(\ve t)$.  In particular, there exists   a unique Schwartz  function  $\vfi\in\eS(\bR)$ such that
\begin{equation}
w(t)=\vfi(t^2),\;\;\forall t\in\bR.
\label{eq: w-vfi}
\end{equation}
Consider the random  function
\begin{equation}
\bu^\ve(\vec{\theta})=\sum_{\vec{k}\in\bZ^m} X^\ve_{\vec{k}}\Psi_{\vec{k}}(\vec{\theta}),
\label{uve}
\end{equation}
where the coefficients $X^\ve_{\vec{k}}$ are independent  Gaussian random variables with mean $0$ and variances
\[
\var(X_{\vec{k}})= w_\ve(2\pi |\vec{k}|)=w(2\pi\ve|\vec{k}|). 
\]
For $\ve>0$ sufficiently small  the random function $\bu^\ve$ is  almost surely (a.s.) smooth and Morse.  By \emph{energy landscape}  of $\bu^\ve$ we  understand the catalogue containing the  basic information about the critical points of $\bu^\ve$: their location, their indices, and their corresponding critical values.   A first information about the energy landscape  concerns the number of the critical points.  Let $N(\bu^\ve)$ denote  the number of critical points  of $\bu^\ve$.  We denote   by $N_\ve$ the  expectation the  random variable $N(\bu^\ve)$,
\[
N_\ve :=\bsE\Bigl(\; N(\bu^\ve)\;\Bigr),
\]
and by $\var_\ve$ its variance
\[
\var^\ve=\bsE\Bigl(\, \bigl(\,N(\bu^\ve)-N_\ve\,\bigr)^2\,\Bigr)= \bsE\bigl(\, N(\bu^\ve)^2\,\bigr) -N_\ve^2.
\]
We  are interested in the the small $\ve$ behavior of the random variable $N_\ve$.  

To understand the analytic significance of the $\ve\to 0$ limit it is best to     think of the special case when $w$  ``approximates''  the characteristic function of the interval $[-1,1]$, i.e., it is  supported  in $[-1,1]$  and  it is identically $1$ in a neighborhood of $0$.   For fixed $\ve$, $\bu^\ve$ is a  trigonometric polynomial: the terms  corresponding to $|\vec{k}|>\ve$  do not contribute to $\bu^\ve$.    If we formally let $\ve\searrow 0$ in (\ref{uve}) , then we deduce that $\bu^\ve$ converges to the random Fourier series
\[
\sum_{\vec{k}\in\bZ^m} X^0_{\vec{k}}\Psi_{\vec{k}}(\vec{\theta}),
\]
where the coefficients $X^0_{\vec{k}}$ are independent, standard normal random variables with  mean $0$ and variance $1$.  

The  above  series is not convergent to any function in any meaningful way but, as explained in \cite{GeVi2},  it converges almost surely  in the sense of  distributions to a random  distribution on the torus, the  Gaussian white noise.   For $\ve>0$ very small,  the  trigonometric polynomial $\bu^\ve$ is highly oscillatory  and we expect it to have many critical points, i.e., $N_\ve\to\infty$ as $\ve\searrow  0$ .  In  \cite{Ncv2}  proved a more  precise statement,  namely 
\begin{equation}
N_\ve\sim  C_m(w) \ve^{-m}\bigl(1+O(\ve)\,\bigr)\;\;\mbox{as $\ve\searrow 0$},
\label{nve}
\end{equation}
where $C_m(w)$  is a certain  explicit constant positive  constant that depends only on $m$ and $w$.   In this paper we  describe the small $\ve$ asymptotics for  the variance of  the normalized random variable $\frac{1}{N_\ve} N(\bu^\ve)$. In particular,  we prove that this random variable  is highly concentrated around its mean.

\subsection{The main result} The goal of this paper is  an asymptotic formula (as $\ve\searrow 0$)  for  the variance of  the random variable  $N(\,\bu^\ve\,)$. 

\begin{theorem}  There exists a constant $C'_m(w)\geq 0$    such that
\begin{equation}
\var^\ve\sim C' _m(w) \ve^{-m}\bigl(\, 1+O(\ve^N)\,\bigr),\;\;\forall N>0,\;\;\mbox{as $\ve\searrow 0$}.
\label{eq: main}
\end{equation}
\label{th: main}
\end{theorem}
Consider the normalized random variables
\[
\widehat{N}(\bu^\ve):=\frac{1}{N_\ve}N(\bu^\ve)=\frac{1}{\bsE(\,N(\bu^\ve)\,)}N(\bu^\ve).
\]

\begin{corollary}  If $(\ve_k)$ is a sequence of positive numbers such that $\sum_k \ve_k^m <\infty$, then   sequence of random  variables $\widehat{N}_{\ve_k}$ converges almost surely  to $1$.
\end{corollary}

\begin{proof} The equalities  (\ref{nve}) and (\ref{eq: main}) imply that the    variance  $\widehat{\var}^\ve$ of  $\widehat{N}(\bu^\ve)$ is extremely small,
\begin{equation}
\widehat{\var}^\ve \sim\frac{C'_m(w)}{C_m(w)^2}\ve^m\;\;\mbox{as $\ve\searrow 0$}.
\label{eq: main2}
\end{equation}
Now conclude using  Chebysev's inequality  and the Borel-Cantelli's lemma.
\end{proof}

The constant $C_m'(w)$ has an explicit, albeit very complicated description.    Here is the gist of it.

Denote by $\Sym_m$ the space of real symmetric  $m\times m$ matrices.  The group $O(m)$ acts on $\Sym_m$ by conjugation and we denote by $\eG_m$ the space of $O(m)$-invariant  Gaussian measures on $\Sym_m$.  This is a $2$-dimensional convex cone  described explicitly in Appendix \ref{s: gmat}.     Then
\[
C_m(w)= \frac{1}{(2\pi d_m)^{\frac{m}{2}}}\left( \int_{\Sym_m} |\det A| \, \bGamma_{h_m,h_m}(dA)\right),
\]
where $h_m,d_m$ are certain momenta  of $w$ defined in (\ref{eq: sdh}) and the Gaussian measure $\bGamma_{h_m,h_m}\in\eG_m$ is described in (\ref{eq: gamauv}).

 For any $\eta\in\bR^m\setminus 0$ we denote by $O_\eta(m)$ the subgroup of $O(m)$ consisting of orthogonal maps which fix $\eta$.  We  denote by  $\eG_{m,\eta}$ the space of $O_\eta(m)$-invariant  Gaussian measures on $\Sym_m$. This is a $5$-dimensional convex cone described   explicitly in Appendix \ref{s: gmat}. We set $\Sym_m^{\times 2}=\Sym_m\oplus \Sym_m$ so that the elements in $B=\Sym_m^{\times 2}$  have the form $B=B^-\oplus B^+$, $B^\pm\in \Sym_m$.  
 
 The  constant $C'_m(w)$ is expressed  in terms of a  family of  Gaussian random matrices 
 \[
 B_\eta=B_\eta^-\oplus B_\eta^+\in\Sym_m^{\times 2}, \;\; \eta\in \bR^m\setminus 0. 
 \]
We denote by $\bbom^0(\eta)$ the distribution of $B_\eta$.   The Gaussian measure  $\bbom^0(\eta)$  has an explicit but quite complicated  description detailed in Appendix \ref{s: b}. 

 The  correspondence $\eta\mapsto\bbom^0(\eta)$   is equivariant with respect  to the action of $O(m)$ on $\bR^m$ and its diagonal action  on the space  of Gaussian  measures on $\Sym_m^{\times 2}$. The components $B_\eta^\pm$ are identically distributed, and their distributions are certain explicit Gaussian measures in $\eG_{m,\eta}$. These components \emph{are not independent},  but become less and less correlated  as $|\eta|\to\infty$. In fact, as $|\eta|\to \infty$ the Gaussian measure $\bbom^0(\eta)$ converges to the  product of the Gaussian measures $\bGamma_\infty=\bGamma_{h_m,h_m}\times \bGamma_{h_m,h_m}$.   We denote by $\bsE_{\bbom^0(\eta)}(|\det B|)$ the expectation of the   random variable $|\det B_\eta|$. 
 
 To each $\eta\in\bR^M\setminus 0$ we  associate  the symmetric $2m\times 2m$ matrix 
 \begin{equation}
 \eH(V, \eta)=\left[\begin{array}{cc}
 -\Hess(V,0)& -\Hess(V,\eta)\\
 -\Hess(V,\eta) &-\Hess(V,0)
 \end{array}
 \right]
 \label{hfeta}
\end{equation}
 where $\Hess(V,\eta)$ denotes the Hessian  at $\eta$ of the function $V:\bR^n\to\bR$,
 \[
 V(\xi)=\int_{\bR^m} e^{-\ii(\xi,x)} w(|x|)|dx|.
 \]
 We set
 \[
 K(\eta):=\frac{1}{\sqrt{\det 2\pi \eH(V,\eta)}}.
 \]
 Then
\[
 K(\eta)\sim const. |\eta|^{-m}\;\;\mbox{as $\eta\to 0$},
 \]
  and as $\eta\to \infty$  the  quantity $K(\eta)$ approaches rapidly  the constant 
  \[
  K(\infty)=\frac{1}{\det \bigl(\, 2\pi\Hess(V,0)\,\bigr)}.
  \]
  Then
 \[
 C_m'(w) = C_m(w) +\int_{\bR^m} \Bigr(\, K(\eta)\bsE_{\bbom^0(\eta)}\bigl(\,|\det B|\,\bigr)- K(\infty)\bsE_{\bGamma_\infty}\bigl(\,|\det B|\,\bigr)\,\Bigr) |d\eta|.
 \]
Let us point out that
\[
C_m(w)^2= K(\infty)\bsE_{\bGamma_\infty}\bigl(\,|\det B|\,\bigr).
\]
 The one-dimensional investigations  in \cite{GW, Ntrig} suggest  that $C_m'(w)>0$, but  all our attempts to proving this  were fruitless so far.

\subsection{Outline of the proof of Theorem \ref{th: main}.} \label{s: outline}
To help the  reader  better navigate the   many technicalities involved  in proving (\ref{eq: main}) we  describe in this section  the  bare-bones  strategy used in the proof of Theorem \ref{th: main}. 

Denote by $\bsD$ the  diagonal of $\bT^m\times \bT^m$
\[
\bsD=\bigl\{\,\Theta= (\vec{\theta},\vec{\vfi})\in\bT^m\times \bT^m;\;\;\vec{\theta}=\vec{\vfi}\,\bigr\}.
\]
We denote by  $\eN_{\bsD}$ the normal bundle of the embedding  $\bsD\hra \bT^m\times \bT^m$. For each $r>0$ sufficiently small   we denote by $\eT_r(\bsD)$ the tube of radius $r$ around $ \bsD$,
\[
\eT_r(\bsD)= \Bigl\{  \Theta\in\bT^m\times \bT^m;\;\;\dist\bigl(\,  \Theta,\bsD\,\bigr) < r\,\Bigr\},
\]
where $\dist$ denotes the geodesic   distance on  $\bT^m\times \bT^m$ equipped with the  product metric $g\times g$.

Denote by $B_r(\bsD)\subset \eN_{\bsD}$ the   radius $r$-disk bundle.  For $\hbar>0$  sufficiently small, the exponential map induces a diffeomorphism
\[
\exp: B_\hbar (\bsD)\to\eT_\hbar (\bsD).
\]
Fix such a $\hbar$. For $t>0$ we denote by $\eR_t:\eN_{\bsD} \to\eN_{\bsD}$ the rescaling by a factor of $t$ along the fibers of the normal bundle $\eN_{\bsD}$.  Thus, $\eR_t$ acts as multiplication by $t$ on each fiber of $\eN_{\bsD}$.

We can now   explain the  strategy.

\medskip

\noindent {\bf Step 1.} Using the Kac-Rice formula we produce a density $\rho_1^\ve(\vec{\theta})|d\vec{\theta}|$ on $\bT^m$ such that
\begin{equation}
N_\ve= \int_{\bT^m}  \rho_1^\ve(\vec{\theta})|d\vec{\theta}|.
\label{eq: KR1}
\end{equation}
Set 
\[
\tilde{\rho}_1^\ve\bigl(\,\Theta\,\bigr):=  \rho_1^\ve(\vec{\theta}) \rho_1^\ve(\vec{\vfi}).
\]
{\bf Step 2.}  Using the Kac-Rice formula we produce a density $\rho_2^\ve(\Theta)|d\Theta|$ on $\bT^m\times \bT^m\setminus \bsD$ such that
\begin{equation}
\underbrace{\bsE\bigl(\, N(\bu^\ve)^2- N(\bu^\ve)\,\bigr)}_{=:\mu^\ve_{(2)}}=\int_{\bT^m\times \bT^m\setminus \bsD} \rho_2^\ve(\Theta)|d\Theta|.
\label{eq: KR2}
\end{equation}
As an aside, let us point out that the ratio
\[
\frac{\rho_2^\ve(\Theta)}{\tilde{\rho}_1^\ve\bigl(\,\Theta\,\bigr)}=\ve^{2m}\frac{\rho_2^\ve(\Theta)}{C_m(w)^2}
\]
is the  so called \emph{two-point correlation function} of the set of critical points of the random function $\bu^\ve$.

The second combinatorial moment $\mu^\ve_{(2)}$ of $N(\bu^\ve)$ is related to the variance $\var^\ve$ via the equality
\begin{equation}
\var^\ve=\mu^\ve_{(2)}-N_\ve^2 + N_\ve.
\label{eq: mom22}
\end{equation}
In view of (\ref{nve}) the  asymptotic estimate (\ref{eq: main}) is equivalent to the existence  of a  real constant $c$ such that
\begin{equation}
\lim_{\ve\searrow 0} \ve^m\bigl(\,\mu_{(2)}^\ve-N_\ve^2\,\bigr)=c.
\label{eq: main3}
\end{equation}
Set
\[
\delta^\ve(\Theta):=\rho_2^\ve(\Theta)- \tilde{\rho}_1^\ve(\Theta).
\]
 Using (\ref{eq: KR1}) and (\ref{eq: KR2}) we   can rewrite
\begin{equation}
\mu_{(2)}^\ve-N_\ve^2= \underbrace{\int_{\eT_\hbar(\bsD)\setminus \bsD} \delta^\ve(\Theta) |d\Theta|}_{=:\bsI_0(\ve)}+\underbrace{\int_{\bT^m\times\bT^m\setminus \eT_\hbar(\bsD)} \delta^\ve(\Theta) |d\Theta|}_{=:\bsI_1(\ve)}.
\label{eq: diff-mom}
\end{equation}
{\bf Step 3. Off-diagonal estimates.}   We show that
\begin{equation}
\lim_{\ve\searrow 0} \ve^{-N} \bsI_1(\ve)=0,\;\;\forall  N>0.
\label{eq: lim1}
\end{equation}
\noindent {\bf Step 4. Near-diagonal estimates.} Denote by $\beta_\ve$ the composition
\[
\beta_\ve: B_{\hbar/\ve}(\bsD)\stackrel{\eR_\ve}{\Lra} B_\hbar(\bsD)\stackrel{\exp}{\Lra} \eT_\hbar(\bsD).
\]
Then
\[
\ve^m \bsI_0(\ve)= \int_{B_{\hbar/\ve}(\bsD)\setminus \bsD}\delta^\ve\bigl(\beta_\ve(\eta)\,\bigr) |d\eta|,
\]
and we  will show that the limit
\begin{equation}
 \lim_{\ve\searrow 0}\int_{B_{\hbar/\ve}(\bsD)\setminus \bsD}\delta^\ve\bigl(\beta_\ve(\eta)\,\bigr) |d\eta|
\label{eq: lim0}
\end{equation}
exists and it is finite. It boils down to showing that the function
\[
\hat{\delta}^\ve:\eN_{\bsD}\setminus \bsD\to \bR,\;\;\hat{\delta}^\ve(\eta)=\begin{cases}
\delta^\ve\bigl(\beta_\ve(\eta)\,\bigr), & \eta\in B_{\hbar/\ve}(\bsD)\setminus \bsD\\
0, &\mbox{otherwise},
\end{cases}
\]
converges as $\ve\searrow 0$ to  an \emph{integrable} function $\hat{\delta}^0:\eN_{\bsD}\to\bR$ and
\[
\lim_{\ve\searrow 0} \int_{\eN_{\bsD}}\hat{\delta}^\ve(\eta)|d\eta|= \int_{\eN_{\bsD}}\hat{\delta}^0(\eta)|d\eta|.
\]
This last step is the most challenging part of the proof.      A big part of the challenge is the fact that $\hat{\delta}^\ve(\eta)$  has a singularity  along the  zero section of $\eN_{\bsD}$.

At this point  we think it is useful  to  provide  a few more details to give  the reader a better sense of the complexity of the problem. Define a new random function
\[
\bsU^\ve:\bT^m\times \bT^m\to \bR,\;\;\bsU_\ve(\vec{\theta},\vec{\vfi})=\bu^\ve(\vec{\theta})+\bu^\ve(\vec{\vfi}). 
\]
We denote  by $N(\bsU_\ve)$ the number of critical points of $\bsU_\ve$ situated  outside the diagonal. Note that
\[
N(\bsU^\ve)= N(\bu^\ve)^2-N(\bu^\ve)
\]
and thus
\[
\bsE\bigl(\, N(\bsU^\ve)\,\bigr)=\mu^\ve_{(2)}. 
\]
The Kac-Rice formula, detailed in Section \ref{s: kac-rice}, shows that 
\begin{equation}
\bsE\bigl(\, N(\bsU^\ve)\,\bigr)=\int_{\bT^m\times \bT^m\setminus \bsD}\;\underbrace{ \frac {1}{\sqrt{ \det  2\pi\widetilde{\bsS}_\ve(\Theta)}} \bsE\Bigl(\, \bigr|\det \Hess \bsU^\ve(\Theta)\bigr|\,\Bigl|\, d\bsU^\ve(\Theta)=0\,\Bigr)}_{=:\rho_2^\ve(\Theta)}\, |d\Theta|,
\label{eq: combmom}
\end{equation}
where $\widetilde{\bsS}_\ve(\Theta)$ is a  symmetric, positive semidefinite $2m\times 2m$ matrix  describing the covariance form of the  random vector $d\bsU^\ve(\Theta)$, and $\bsE(X|Y=0)$ denotes the conditional expectation of the random variable $X$ given that $Y=0$.    

The matrix  $\widetilde{\bsS}_\ve(\Theta)$ becomes singular along the diagonal because the two components $d\bu^\ve(\vec{\theta})$ $d\bu^\ve(\vec{\vfi})$ become \emph{dependent}  random  vectors for $\vec{\theta}=\vec{\vfi}$.     What is worse, one can show that   $\det  \widetilde{\bsS}_\ve(\Theta)$   behaves like $\dist(\Theta,\bsD)^{2m}$ near $\bsD$. Hence the term
\[
\Theta\mapsto \frac {1}{\sqrt{ \det 2\pi \widetilde{\bsS}_\ve(\Theta)}}
\]
is not integrable near the diagonal.  For our strategy to work  we need that the second term 
\[
\Theta\mapsto \bsE\Bigl(\, \bigr|\det \Hess \bsU^\ve(\Theta)\bigr|\,\Bigl|\, d\bsU^\ve(\Theta)=0\,\Bigr)
\]
 vanish along the diagonal $\bsD$.   Let us give a heuristic argument why this is to be expected.  
 
The mean value theorem shows that 
\[
\frac{1}{|\vec{\vfi}-\vec{\theta}|}\bigl(\,d\bu^\ve(\vec{\vfi}) -d\bu^\ve(\vec{\theta})\,\bigr)= \frac{1}{|\vec{\vfi}-\vec{\theta}|}\int_0^{|\vec{\vfi}-\vec{\theta}|} \Hess \bu^\ve( \vec{\theta}+t\eta)\eta \, dt,
\]
where
\[
\eta=\eta(\Theta) :=\frac{1}{|\vec{\vfi}-\vec{\theta}|}(\vec{\vfi}-\vec{\theta}).
\]
 If we take into account the condition $d\bsU^\ve(\vec{\theta},\vec{\vfi})=0$, i.e., $d\bu^\ve(\vec{\vfi}) = d\bu^\ve(\vec{\theta})=0$, then we deduce
\[
\frac{1}{|\vec{\vfi}-\vec{\theta}|}\int_0^{|\vec{\vfi}-\vec{\theta}|} \Hess \bu^\ve( \vec{\theta}+t\eta)\eta dt=0.
\]
If we let $\vec{\vfi}\to\vec{\theta}$,  so that $\eta(\vec{\theta},\vec{\vfi})$  stays fixed,  i.e., $\Theta$ approaches the diagonal along a fixed direction given by the unit vector $\eta$,   we deduce  that    the linear operator $\Hess \bu^\ve( \vec{\theta}+t\eta)$ admits in the limit $t\searrow 0$  a one-dimensional kernel.  In particular, the Hessian $\Hess \bsU^\ve(\Theta)$ conditioned   by requirement $d\bsU^\ve(\Theta)=0$, ought to approach a linear operator with  a two dimensional kernel because 
\[
\Hess \bsU^\ve(\Theta)=\Hess \bu^\ve(\vec{\theta})\oplus \Hess \bu^\ve(\vec{\vfi}).
\]
We do not know how to  transform this heuristic argument into a rigorous one, but we can prove by analytic means  that near the diagonal  we have (see (\ref{eq: resc_cond_hess}) and (\ref{eq: exp-growth}) )
\[
 \bsE\Bigl(\, \bigr|\det \Hess \bsU^\ve(\Theta)\bigr|\,\Bigl|\, d\bsU^\ve(\Theta)=0\,\Bigr) \sim const.\dist(\Theta,\bsD)^2.
 \]
This  guarantees that the integrand $\rho_2^\ve(\Theta)$ in (\ref{eq: combmom}) is integrable  near the diagonal because
\begin{equation}\label{2pointcorr}
\rho_2^\ve(\Theta)\sim const.\dist(\Theta,\bsD)^{2-m}.
\end{equation}
 With a bit of work one can show that the integrand does not explode anywhere away from the diagonal so  the integral  in (\ref{eq: combmom}) is finite. There is an added complication because we are actually interested in the  singular limit $\ve\searrow 0$ so that  we need to produces estimates that are uniform   in $\ve$ small.  Alas, this  is  not the only obstacle.

The arguments  described above  lead  to the conclusion that 
\[
\mu_{(2)}^\ve\sim C_m(w)^2\ve^{-2m},
\]
where $C_m(w)$ is the constant in (\ref{nve}). On the other hand $N_\ve\sim  C_m(w)\ve^{-m}$   so that
\[
\mu_{(2)}^\ve-N_\ve^2=o(\ve^{-2m}).
\]
To prove Theorem  \ref{th: main} we  need to substantially improve this  estimate to an estimate  of the  form
\[
\mu_{(2)}^\ve-N_\ve^2\sim Z\ve^{-m}
\]
for some real constant $Z$. This is where the usage of the  singular rescaling maps $\beta_\ve$   saves the day.

\subsection{Related results}  There  has been considerable work on the statistics of zero sets  of random functions or sections.  The  simplest invariant of such a set is its volume. In particular, if the zero set is zero dimensional, its volume  coincides with its cardinality.  The Kac-Rice  formula leads rapidly  to    information  about the expectation of the volume   of such a random set.    Higher order information, such as the variance it is typically harder to obtain.

In the complex  case   B. Shiffman and S. Zeldich  \cite{SZ1,SZ2} have obtained rather precise information  on the variance of the number  of simultaneous zeros    of $m$ independent random sections of a positive  holomorphic line bundle over an $m$-dimensional K\"{a}hler manifold.

The statistics of  the zero set of a random eigenfunction  of the Laplacian on a flat torus  were investigated by Z. Rudnick and I. Wigman in \cite{RW}. In particular, in this paper the authors  produce upper estimates on the variance of the volume of the zero set of such a random eigenfunction   leading to  concentration results very similar to the one  we obtain in the present paper.  In  the case of two-dimensional tori, M. Krishnapur, P. Kurlberg and I. Wigman \cite{KKW} have refined  the above upper estimate to a precise  asymptotic estimate.    

The statistics of the volume of the zero set of a random eigenfunction on the round $m$-dimensional sphere were investigated  by I. Wigman in \cite{W1}.  In particular, he proves    upper estimates for the variance  leading to concentration results.        In the paper \cite{W2}  he  consider the special case of the $2$-sphere and  describes  exact asymptotic estimates   for the variance of the volume of the zero set of a random eigenfunction corresponding to a large eigenvalue.

Recently, Bleher, Homma and Roeder  \cite{BHR} proved a counterpart  of (\ref{2pointcorr})  for  the two-point correlations functions   determined by  the solutions of random \emph{real}  polynomial systems of several \emph{real} variables. 

A different  different type of concentration result  is discussed  by E. Subag in \cite{Sub}, where the author investigates  the behavior  of the energy  landscapes of certain random functions on the round sphere $S^n$ as $n\to\infty$. 

\subsection{Organization of the paper} In the short Subsection \ref{s: kac-rice}    we present the  version of the Kac-Rice formula we use to compute  expectations of the number of critical points of various random functions.  The rest of  Section  \ref{s: rho1} contains  a more detailed  analysis of the correlation kernel of $\bu^\ve$ together with an explicit decryption of the density $\rho_1^\ve$ that computes $N_\ve$.

Section \ref{s: rho2} is devoted to the computation of the  density $\rho_2^\ve$ on $\bT^m\times \bT^m\setminus \bsD$. This  density   is expressed  in terms  of two quantities.

\begin{itemize}

\item The covariance kernel of the random field defined by the differential of the random function
\[
\bsU^\ve: \bT^m\times\bT^m\to\bR, \;\;\bsU^\ve(\vec{\theta},\vec{\vfi})=\bu^\ve(\vec{\theta})+\bu^\ve(\vec{\vfi}).
\]
\item The   conditional Hessian of $\bsU^\ve$ given that  $d\bsU^\ve=0$.

\end{itemize}

The Gaussian random vector $d\bsU^\ve(\vec{\theta},\vec{\vfi})$ degenerates along the diagonal and  in  Subsections \ref{ss: 53}, \ref{ss: 54} we    investigate in great detail  this degeneration.  The  statistics of the above conditional Hessians are described in  Subsection \ref{ss: 55}.  Suitably rescaled, these conditional Hessians  have a limit as $\ve\to 0$. This limit is described in Appendix \ref{s: b}.  In  Subsection \ref{ss: 56}   we  give an explicit description of the density $\rho_2^\ve$.  The behavior of this density away from the diagonal  is described  in Subsection \ref{ss: 61}. The   behavior of $\rho_2^\ve$ near the diagonal is investigated in Subsection \ref{ss: 62}, \ref{ss: 63}.   We complete the proof of Theorem \ref{th: main} in Subsection \ref{ss: 64}.

Throughout the paper  we use frequently  the following special case of  Proposition \ref{prop: diff-gauss} in Appendix \ref{s: a}: if $V$ is a finite dimensional Euclidean space, $\Sym(V)$ is the space of symmetric operators on $V$ and $\eG$ is the set of centered Gaussian measures on $\Sym(V)$, then the map 
\[
\eG\ni \Gamma\mapsto \bsE_\Gamma(|\det A|)\in \bR
\]
is locally H\"{o}lder continuous with H\"{o}lder exponent $\frac{1}{2}$. The various Gaussian ensembles of  symmetric matrices  that appear   in the proof are described  in great detail  in Appendix \ref{s: gmat}.

\section{The density $\rho_1^\ve$}
\label{s: rho1}
\setcounter{equation}{0}

\subsection{The Kac-Rice formula}\label{s: kac-rice}

For the reader's convenience  we   give a brief  description of the Kac-Rice formula  used in the proof of Theorem \ref{th: main}. For  proofs and many more details we refer to \cite{AT, AzWs}.

Suppose that $(X,g)$ is a smooth, connected Riemann manifold of dimension $n$ and $u:X\to \bR$ is a Gaussian random  function with covariance kernel
\[
\eE:X\times X\to\bR,\;\;\eE(p,q)= \bsE\bigl(\,u(p)\cdot u(q)\,\bigr).
\]
Under certain explicit conditions on $\eE$  (satisfied  in the  the situations we are interested in) the random function $u$ is almost surely (a.s.)  smooth and Morse.       Assume  therefore that $u$ is  a.s. smooth and Morse.   For any  precompact open  set $\eO\subset X$ we denote by $N(u, \eO)$ the number of critical  points  of $u$ in $\eO$. This is a random variable   whose expectation $N(\eO)$ is given by the Kac-Rice formula. 

To state this formula we need to introduce a bit of terminology.  Fix a point $p\in X$ and  normal coordinates $(x^1,\dotsc, x^n)$ at $p$ so that $x^i(p)=0$, $\forall i$.  Thus in the neighborhood of $(p,p)\in X\times X$ we can view $\eE$ as a function of two (sets of) variables $\eE=\eE(x,y)$.

The differential  of $u$ at $p$ is a Gaussian random  vector $du(\bp)\in T^*_pX$   described by its covariance form
\[
S_p(du): T_\bp X\times T_p X\to \bR.
\]
This is  a   symmetric nonnegative  definite  form  described   in the  chosen coordinates by the $n\times n$ matrix $(S(du(p) )_{ij})$, where
\[
S_p(du)_{ij}= \bsE\bigl(\,\pa_{x^i}u(p)\pa_{x^j}u(p)\,\bigr)= \frac{\pa^2}{\pa x^i\pa y^j}\eE(x,y)|_{x=y=0}.
\]
We will assume that $S_p(du)$ is actually \emph{positive} definite.

The  Hessian of $u$ at $p$  is the linear operator  $\Hess u(p): T_p M\to T_pM$ which in the above coordinates is described by the symmetric matrix $(\Hess_{ij} (p))_=(\pa^2_{x^ix^j}u(p)\,)_{1\leq i,j\leq n}$. It is a Gaussian random  symmetric matrix described by the covariances
\[
 \bsE\bigl(\,\Hess_{ij}(p)\cdot\Hess_{k\ell}(p)\,\bigr)= \frac{\pa^4}{\pa x^i\pa x^j\pa y^k \pa y^\ell}\eE(x,y)_{x=y=0}.
 \]
  The \emph{Kac-Rice formula}    states that
\[
N(\eO)=\int_{\eO} \rho_u(p) |dV_g(p)|,
\]
where
\[
\rho_u(p)=\frac{1}{\sqrt{\det 2\pi S_p(du)}} \bsE\Bigl(\, \;\bigr|\,\det \Hess u(p)\,\bigr|\;\Bigr|\;  du(p)=0\,\Bigr).
\]
The \emph{regression formula}   reduces the  computation of the above conditional expectation to the computation of  an  absolute expectation $\bsE( |\det A_p|\,)$ where $A_p$ is  another  Gaussian random symmetric matrix. The Gaussian distribution  governing this new random   matrix can be expressed explicitly  in terms  of the covariance form of  $\Hess u(p) $ and the correlations between $du(p)$ and $\Hess u(p)$.

\subsection{The covariance kernel of ${\bu}^\ve$.} A simple computation shows that the covariance kernel of the random function  $\bu^\ve$ is
\begin{equation}
\eE^\ve(\vec{\theta},\vec{\vfi}):=\bsE\bigl(\,\bu^\ve(\theta)\cdot\bu^\ve(\vfi)\,\bigr)=\sum_{\vec{k}\in\bZ^m }w(2\pi\ve|\vec{k}|)e^{-2\pi\ii\lan\vec{k}, \vec{\vfi}-\vec{\theta}\ran}. 
\label{eq: cov0}
\end{equation}
Set $\vec{\tau}:=\vec{\vfi}-\vec{\theta}$ and define $\phi:\bR^m\to\bC$ by
\[
\phi(\vec{x}):=e^{-\ii\lan\vec{x},\frac{1}{\ve}\vec{\tau}\ran}w(|\vec{x}|).
\]
We deduce that
\[
\eE^\ve(\vec{\theta}, \vec{\vfi})=\sum_{\vec{k}\in\bZ^m}  \phi(2\pi\ve\vec{k}). 
\]
Using Poisson formula \cite[\S 7.2]{H1} we deduce  that for any $a>0$ we have
\[
\sum_{\vec{k}\in\bZ^m}\phi\left(\frac{2\pi}{a}\vec{k}\right)=\left(\frac{a}{2\pi}\right)^m \sum_{\vec{\nu}\in\bZ^m}\widehat{\phi}(a\vec{\nu}), 
\]
where  for any $u\in\eS(\bR^m)$ we denote by $\widehat{u}(\xi)$ its Fourier transform
\[
\widehat{u}(\xi)=\int_{\bR^m} e^{-\ii\lan\xi,\vec{x}\ran} u(\vec{x})|d\vec{x}|. 
\]
If we let  $a=\ve^{-1}$, then we deduce
\[
\eE^\ve(\vec{\theta},\vec{\vfi})=\frac{1}{(2\pi\ve)^m} \sum_{\vec{\nu}\in\bZ^m}\widehat{\phi}\left(\ve\vec{\nu}\right). 
\]
Define $V:\bR^m\to \bR$ by
\begin{equation}\label{vxi}
V(\xi): =\int_{\bR^m} e^{-\ii\lan\xi,\vec{x}\ran} w(|\vec{x}|)\, |d\vec{x}|.
\end{equation}
Then
\[
\widehat{\phi}(\xi)=V\Bigl(\;\xi+\frac{1}{\ve}\vec{\tau}\;\Bigr)=V\Bigl(\;\xi+\frac{1}{\ve}(\vec{\vfi}-\vec{\theta})\,\;\Bigr). 
\]
Hence
\[
\eE^\ve(\vec{\theta},\vec{\vfi})= \frac{1}{(2\pi\ve)^m}\sum_{\vec{\nu}\in\bZ^m}V\left(\,\frac{1}{\ve}\vec{\tau}+\frac{1}{\ve}\vec{\nu}\right).
\]
We set
\begin{equation}
V^\ve(\vec{\theta}):=\sum_{\vec{\nu}\in\bZ^m} V\left(\,\vec{\theta}+\frac{1}{\ve}\vec{\nu}\,\right).
\label{eq: vve}
\end{equation}
We deduce that
\begin{equation}
\eE^\ve(\vec{\theta},\vec{\vfi}) = \frac{1}{(2\pi\ve)^m} V^\ve\left(\frac{1}{\ve}\tau\right).\label{eq: cov-asy0}
\end{equation}
From the special  form (\ref{eq: vve}) of $V^\ve$ and the  fact that $V$ is a Schwartz function we deduce  that for any positive integers $k, N$ and any $R>0$ we have
 \begin{equation}
 \Vert V^\ve-V\Vert_{C^k(B_R(0))}=O(\ve^N)\;\;\mbox{as $\ve\searrow 0$},
 \label{eq: vve_appr}
 \end{equation}
where $B_R(0)$ denotes the  open ball of radius $R$  in $\bR^M$ centered at the origin.

\begin{remark} We define
 \begin{equation}
 \begin{split}
 s_m:=\int_{\bR^m} w(|x|) dx,\;\;d_m:= \int_{\bR^m} x_1^2w(|x|) dx,\\
 h_m:= \int_{\bR^m} x_1^2x_2^2w(|x|) dx.
  \end{split}
 \label{eq: sdh}
 \end{equation}
For any multi-index $\alpha$ we have
  \[
 \int_{\bR^m}w(|x|)x^{\alpha} dx =\left(\int_{|x|=1} x^{\alpha} dA(x)\right)\underbrace{\left(\int_0^\infty w(r) r^{m+|\alpha|-1} dr\right)}_{=:I_{m,|\alpha|}(w)}.
 \]
 On the other hand, according to \cite[Lemma 9.3.10]{N0} we have
 \begin{equation}
\int_{|x|=1} x^{\alpha} dA(x)=Z_{m,\alpha}:=\begin{cases}
\frac{2\prod_{i=1}^k\Gamma(\frac{\alpha_i+1}{2})}{\Gamma(\frac{m+|\alpha|}{2})}, &\alpha\in (2\bZ_{\geq 0})^m,\\
0, & {\rm otherwise.}
\end{cases}
\label{eq: zmab}
\end{equation}
Note that 
\[
\int_{\bR^m}w(|\vec{x}|) x_1^2x_2^2 \,|d\vec{x}|=\Gamma\left(\frac{3}{2}\right)^2\frac{2\prod_{k=3}^m\Gamma(\frac{1}{2})}{\Gamma(2+\frac{m}{2})} I_{m,4}(w),
\]
\[
\int_{\bR^m}w(|\vec{x}|) x_1^4\,|d\vec{x}|= \Gamma\left(\frac{5}{2}\right)\frac{2\prod_{k=2}^m\Gamma(\frac{1}{2})}{\Gamma(2+\frac{m}{2})} I_{m,4}(w)=3 \int_{\bR^m}w(|\vec{x}|) x_1^2x_2^2 \,|d\vec{x}|.
\]
We deduce from the above computations that
\begin{subequations}
\begin{equation}
\pa^2_{\xi_i\xi_j}V(0)= - \int_{\bR^m}x_ix_j w(|\vec{x}|)\,|d\vec{x}|= -d_m\delta_{ij}
\label{eq: deriv1}
\end{equation}
\begin{equation}
\pa^2_{\xi_i\xi_j\xi_k\xi_\ell}V(0)=\int_{\bR^m}x_ix_j x_kx_\ell w(|\vec{x}|)\,|d\vec{x}| =h_m(\delta_{ij}\delta_{k\ell}+\delta_{ik}\delta_{j\ell}+ \delta_{i\ell}\delta_{jk}).
\label{eq: deriv2}
\end{equation}
\end{subequations}

\qed
\end{remark}

For  $\ve \geq 0$,  an orthornormal basis $(\be_1,\dotsc,\be_m)$ of $\bR^m$, a multi-index $\alpha\in \bZ_{\geq 0}^m$  and $\eta\in\bR^m$, we set
\begin{equation}
V_\alpha(\eta):=(\pa^{\alpha_1}_{\be_1}\cdots \pa^{\alpha_m}_{\be_m} V)(\eta),\;\;V^\ve_\alpha(\eta):=( \pa^{\alpha_1}_{\be_1}\cdots \pa^{\alpha_m}_{\be_m} V^\ve)(\eta).
\label{eq: VM}
\end{equation}
Note that $V(\xi)$ is radially symmetric and   can be written as $f(|\xi|^2/2)$ ,  for some Schwartz function $f\in\eS(\bR)$,
 \begin{equation}
f\bigl(\, |\xi|^2/2\,\bigr)=V(\xi)= \int_{\bR^m} e^{-\ii\lan\xi,\vec{x}\ran} w(|\vec{x}|)\, |d\vec{x}|.
\label{f}
\end{equation}
We have
\begin{subequations}
\begin{equation}
V_i(\eta)= \eta_i f'\Bigl(\,\frac{|\eta|^2}{2}\,\Bigr),
\label{eq: d1}
\end{equation}
\begin{equation}
V_{i,j}(\eta)= \delta_{ij} f'\Bigl(\, \frac{|\eta|^2}{2}\,\Bigr) +\eta_i\eta_j f''\Bigl(\, \frac{|\eta|^2}{2}\,\Bigr),
\label{eq: d2}
\end{equation}
\begin{equation}
V_{i,j,k}(\eta)= \bigl(\; \delta_{ij}\eta_k+\delta_{ik}\eta_j+\delta_{jk}\eta_i\;\bigr)f''\Bigl(\, \frac{|\eta|^2}{2}\,\Bigr)+\eta_i\eta_j\eta_k f'''\Bigl(\, \frac{|\eta|^2}{2}\,\Bigr),
\label{eq: d3}
\end{equation}
\begin{equation}
\begin{split}
V_{i,j,k,\ell}(\eta) = \bigl(\;\delta_{ij}\delta_{k\ell}+\delta_{ik}\delta_{j\ell}+\delta_{jk}\delta_{i\ell}\;\bigr) f'' \Bigl(\, \frac{|\eta|^2}{2}\,\Bigr) & \\
 + \bigl(\; \delta_{ij}\eta_k\eta_\ell+\delta_{ik}\eta_j\eta_\ell+\delta_{jk}\eta_i\eta_\ell+\delta_{i\ell}\eta_j\eta_k+\delta_{j\ell}\eta_i\eta_k+\delta_{k\ell}\eta_i\eta_j\;\bigr) f'''\Bigl(\, \frac{|\eta|^2}{2}\,\Bigr) &  \\
+\eta_i\eta_j\eta_k\eta_\ell f^{(4)}\bigl(\,\frac{|\eta|^2}{2} \,\bigr) &. 
\label{eq: d4}
\end{split}
\end{equation}
\end{subequations}
 In particular,
\begin{equation}
s_m=f(0),\;\; d_m= -f'(0),\;\; h_m= f''(0).
\label{eq: sdh1}
\end{equation}

\subsection{The Kac-Rice formula for $N_\ve$}    Fix an orthonormal basis $\be_1,\dotsc,\be_m$ of $\bR^m$ and set $\pa_i=\pa_{\be_i}$  and
\[
\bu^\ve_{i_1,\dotsc, i_k}:=\pa_{i_1}\cdots \pa_{i_k}\bu^\ve.
\]
We denote by $\bsS_\ve(\vec{\theta})$ the covariance form of the vector  $d\bu^\ve(\vec{\theta})$, i.e. the  symmetric matrix 
\[
\bsS_\ve(\vec{\theta}) =\Bigl( \bsE\bigl(\, \bu_i^\ve(\vec{\theta})\cdot \bu_j^\ve(\vec{\theta})\,\bigr)\,\Bigr)_{1\leq i,j\leq m}=\Bigl( \,  \pa^2_{\theta_i\vfi_j}\eE^\ve(\vec{\theta},\vec{\vfi})|_{\vec{\theta}=\vec{\vfi}}\,\Bigr)_{1\leq i,j\leq m}.
\]
Using (\ref{eq: cov-asy0}) we deduce
\[
\pa^2_{\theta_i\vfi_j}\eE^\ve(\vec{\theta},\vec{\vfi})|_{\vec{\theta}=\vec{\vfi}}= -(2\pi)^{-m}\ve^{-m-2}V^\ve_{i,j}(0).
\]
Note that
\[
V^\ve_{i,j}(0)=V^0_{i,j}(0) + O(\ve^N),\;\;\forall N.
\]
For any $\eta\in\bR^m$  and any smooth function $F:\bR^m\to \bR$  we denote by $\bsH(F,\eta)$ the Hessian of $F$ at $\eta$, so that
\[
(2\pi)^m \ve^{m+2}\bsS_\ve(\vec{\theta})= \bsH(-V^\ve,0).
\]
We denote by $\si^\ve_{ij}$ the entries of $\bsH(-V^\ve, 0)^{-1}$ and by $\check{s}_{ij}^\ve$ the entries of $\bsS_\ve(\vec{\theta})^{-1}$ so that
\[
\check{s}_{ij}^\ve=(2\pi)^m\ve^{m+2}\si_{ij}^\ve. 
\]
Then the Kac-Rice formula \cite{AT} (together with the explanations in \cite{Ncv2}) show that
\begin{equation}
\begin{split}
N_\ve =\frac{1}{(2\pi)^{\frac{m}{2}}}\int_{\bT^m} \bigl(\det \bsS_\ve(\vec{\theta})\,\bigr)^{-\frac{1}{2}} \bsE\Bigl(\, |\det \bsH(\bu^\ve,\vec{\theta}\,)|\,\bigr|\, d\bu^\ve(\vec{\theta})=0\,\Bigr) |d\vec{\theta}|\\
=\frac{(2\pi)^{\frac{m^2}{2}}\ve^{\frac{m(m+2)}{2}}}{(2\pi)^{\frac{m}{2}}}\int_{\bT^m} \bigl(\det\bsH(-V^\ve,0)\,\bigr)^{-\frac{1}{2}} \bsE\Bigl(\, |\det \bsH(\bu^\ve,\vec{\theta})|\,\bigr|\, d\bu^\ve(\vec{\theta}\,)=0\,\Bigr) |d\vec{\theta}|.
\end{split}
\label{eq: nve1}
\end{equation}
The Hessian  $\bsH(\bu^\ve,\vec{\theta}\,)$ is a  Gaussian random matrix with entries $(x_{ij})$ satisfying the correlation  equalities
\[
\Omega^\ve_{i,j|k,\ell}:=\bsE(x_{ij} x_{k\ell}) =\pa^4_{\theta_i\theta_j\vfi_k\vfi_\ell}\eE^\ve(\vec{\theta},\vec{\vfi})|_{\vec{\theta}=\vec{\vfi}}=(2\pi)^{-m} \ve^{-m-4} V^\ve_{i,j,k,\ell}(0),\;\; 1\leq i,j,k,\ell\leq m.
\]
The  random matrix  $\hat{\bsH}_\ve(\vec{\theta}) $   obtained from  $\Hess_{\vec{\theta}}(\bu^\ve)$  by conditioning  that $d\bu^\ve(\vec{\theta})=0$     is also Gaussian   and  its entries $\hat{x}_{ij}$ satisfy correlation equalities determined by the regression formula
\[
\Xi^\ve_{i,j|k,\ell}=\Omega^\ve_{i,j| k,\ell} -\sum_{a,b=1}^m\bsE\bigl(\, \bu^\ve_{i,j}(\vec{\theta}) \bu^\ve_a(\vec{\theta})\,\bigr)\check{s}^\ve_{ab}\bsE\bigl(\,\bu^\ve_b(\vec{\theta})\bu^\ve_{k,\ell}(\vec{\theta})\,\bigr)
\]
\[
=(2\pi)^{-m}\ve^{-m-4}\left(V^\ve_{i,j,k,\ell}(0)- \sum_{a,b}^m\; \underbrace{V^\ve_{i,j,a}(0)V^\ve_{k,\ell, b}(0) }_{=0}\;\si^\ve_{ab}\right)
\]
\[
=(2\pi)^{-m}\ve^{-m-4} V^\ve_{i,j,k,\ell}(0) =(2\pi)^{-m}\ve^{-m-4}\Bigl(\; V_{i,j,k,\ell}(0) +O\bigl(\ve^N\bigr)\;\Bigr),\;\;\forall N>1.
\]
Denote by $\Gamma_{\Upsilon^\ve}$ the Gaussian  measure on $\Sym_m$   defined by the covariance equalities
\begin{equation}
\Upsilon^\ve_{i,j|k,\ell}:=\bsE(x_{i,j}x_{k,\ell})=(2\pi)^m\ve^{m+4}\Xi^\ve_{i,j| k,\ell}=V^\ve_{i,j,k,\ell}(0).
\label{eq: upsivea}
\end{equation}
Then
\begin{equation}
\bsE\Bigl(\, |\det \bsH(\bu^\ve,\vec{\theta}\,)|\,\bigr|\, d\bu^\ve(\vec{\theta})=0\,\Bigr)  = (2\pi)^{-\frac{m^2}{2}}\ve^{-\frac{m(m+4)}{2}}\int_{\Sym_m}|\det A| \, \Gamma_{\Upsilon^\ve}(dA).
\label{eq: upsive}
\end{equation}
Using (\ref{eq: nve1}) we deduce
\begin{equation}
N_\ve=\frac{\ve^{-m}}{(2\pi)^{\frac{m}{2}}}\int_{\bT^m} \bigl(\det\bsH(-V^\ve,0)\,\bigr)^{-\frac{1}{2}}\left(\, \int_{\Sym_m}|\det A|\, \Gamma_{\Upsilon^\ve}(dA)\,\right)\; |d\vec{\theta}|.
\label{eq: nve2}
\end{equation}
This shows that $\rho_1^\ve(\vec{\theta})$ is independent of  $\vec{\theta}$ and
\[
\rho_1^\ve(\vec{\theta})=\frac{\ve^{-m}}{(2\pi)^{\frac{m}{2}}}  \bigl(\det\bsH(-V^\ve,0)\,\bigr)^{-\frac{1}{2}}\left(\, \int_{\Sym_m}|\det A|\, \Gamma_{\Upsilon^\ve}(dA)\,\right).
\]
We  denote by $\Sym_m^{\times 2}$ the space  of symmetric $2m\times 2m$ matrices $B$ that have a diagonal block  decomposition
\[
B=\left[
\begin{array}{cc}  
B_-  & 0\\
0 & B_+
\end{array}
\right],\;\;B_\pm\in\Sym_m.
\]
The probability measure $\Gamma_{\Upsilon^\ve}$    on $\Sym_m$ induces a probability measure 
\[
\Gamma_{\Upsilon^\ve\times\Upsilon^\ve} :=\Gamma_{\Upsilon^\ve}\otimes \Gamma_{\Upsilon^\ve}
\]
 on $\Sym_m^{\times 2}$. Using the notation in Section \ref{s: outline} we deduce
\begin{equation}
\tilde{\rho}_1^\ve(\vec{\theta},\vec{\vfi}) =\frac{\ve^{-2m}}{(2\pi)^m}  \bigl(\det\bsH(-V^\ve,0)\,\bigr)^{-1}\left(\, \int_{\Sym_m^{\times 2}}|\det B|\, \Gamma_{\Upsilon^\ve\times \Upsilon^\ve}(|dB|)\,\right).
\label{eq: tildro}
\end{equation}
For any smooth function $F:\bR^m\to\bR$ we introduce the symmetric $2m\times 2m$ matrix
\begin{equation}
\eH_\infty(F):=\left[
\begin{array}{cc}
\bsH(-F,0) & 0\\
0 & \bsH(-F,0).
\end{array}
\right].
\label{hinfty}
\end{equation}
We observe that $\det \eH_\infty(F)=\bigl(\, \det \bsH(-F,0)\,\bigr)^2$. In view of this we  deduce
\begin{equation}
\tilde{\rho}_1^\ve(\Theta)=\frac{\ve^{-2m}}{(2\pi)^m\sqrt{\det \eH_\infty(V^\ve)}} \int_{\Sym_m^{\times 2}} |\det B| \, \Gamma_{\Upsilon^\ve\times \Upsilon^\ve}(|dB|),\;\;\Theta=(\vec{\theta},\vec{\vfi}).
\label{eq: tildro1}
\end{equation}

\begin{remark} Observe that
\[
\det \bsH(-V^\ve,0)= \det \bsH(-V,0) +O(\ve^N),\;\;\forall N>0,
\]
and
\[
\det \bsH(-V,0) = d_m^m.
\]
We set 
\begin{equation}
\begin{split}
\Upsilon^0_{i,j|k,\ell}:= V_{i,j,k,\ell}(0)= f''(0)\bigl(\, \delta_{ij}\delta_{k\ell}+ \delta_{ik}\delta_{j\ell}+ \delta_{i\ell}\delta_{jk}\,\bigr)\\
= h_m\bigl(\, \delta_{ij}\delta_{k\ell}+ \delta_{ik}\delta_{j\ell}+ \delta_{i\ell}\delta_{jk}\,\bigr).
\end{split}
\label{eq: upsi00}
\end{equation}
The collection  $\Upsilon^0=\bigl(\Upsilon^0_{i,j|k,\ell}\,\bigr)_{1\leq i,j,k,\ell\leq m}$ describes the covariance form   of an $O(m)$-invariant measure $\Gamma_{\Upsilon^0}$ on $\Sym_m$. Using  the terminology in Appendix \ref{s: gmat} we have
\[
\Gamma_{\Upsilon^0}=\bGamma_{h_m,h_m}.
\]
Observe that
\[
\bigl|\,\Upsilon^\ve_{i,j|k,\ell}-\Upsilon^0_{i,j|k,\ell}\,\bigr|=O(\ve^N)\,\;\;\forall N>0,\;;\forall1\leq i,j,k,\ell\leq m.
\]
Proposition \ref{prop: diff-gauss} implies that
\[
\int_{\Sym_m} |\det A| \Gamma_{\Upsilon^0}(dA)= \int_{\Sym_m} |\det A|\, \bGamma_{h_m,h_m}(dA)+O(\ve^N),\;\;\forall N >0.
\]
Using this in (\ref{eq: nve2}) we deduce
\begin{equation}
N_\ve=\frac{\ve^{-m}}{(2\pi d_m)^{\frac{m}{2}}}\left( \int_{\Sym_m} |\det A| \, \bGamma_{h_m,h_m}(dA)\right)\cdot\bigl(\,1+O(\ve^N)\,\bigr)\;\;\forall N>0.
\label{eq: nve3}
\end{equation}\qed
\label{rem: upsi0}
The constant $C_m(w)$ in (\ref{nve}) is given by
\begin{equation}
\begin{split}
C_m(w)=\frac{1}{(2\pi d_m)^{\frac{m}{2}}}\left( \int_{\Sym_m} |\det A| \, \bGamma_{h_m,h_m}(dA)\right)\\
= \left(\frac{h_m}{2\pi d_m}\right)^{\frac{m}{2}}\left( \int_{\Sym_m} |\det B| \, \bGamma_{1,1}(dB)\right).
\end{split}
\label{nvec}
\end{equation}
\end{remark}

\section{The density $\rho_2^\ve$}
\label{s: rho2}
\setcounter{equation}{0}

\subsection{The covariance kernel of $\bsU_\ve$} The covariance kernel  of $\bsU_\ve$ is the function
\[
{\teE}^\ve(\vec{\theta}_1,\vec{\vfi}_1; \vec{\theta}_2,\vec{\vfi}_2) =\bsE\Bigl(\,\bsU^\ve(\vec{\theta}_1,\vec{\vfi}_1)\bsU^\ve(\vec{\theta}_2,\vec{\vfi}_2)\,\Bigr)
\]
\[
=\eE^\ve(\vec{\theta}_1,\vec{\theta}_2)+\eE^\ve(\vec{\theta}_1,\vec{\vfi}_2)+\eE^\ve(\vec{\vfi}_1,\vec{\theta}_2)+\eE^\ve(\vec{\vfi}_1,\vec{\vfi}_2)
\]
\[
= (2\pi\ve)^{-m}\Biggl( V^\ve\Bigl(\ve^{-1}(\vec{\theta}_2-\vec{\theta}_1)\Bigr)+ V^\ve\Bigl(\ve^{-1}(\vec{\vfi}_2-\vec{\theta}_1)\,\Bigr)
+V^\ve\Bigl(\ve^{-1}(\vec{\theta}_2-\vec{\vfi}_1)\,\Bigr) +V^\ve\Bigl(\ve^{-1}(\vec{\vfi}_2-\vec{\vfi}_1)\,\Bigr)\Biggr).
\]
Let us introduce the notation
\[
\Theta:=(\vec{\theta},\vec{\vfi})\in\bT^m\times\bT^m , \tau(\Theta):=\vec{\vfi}-\vec{\theta},\;\;\tau^\ve=\tau^\ve(\Theta):=\ve^{-1}\tau(\Theta).
\]
We  need to understand the quantities
\[
\pa^\alpha_{\Theta_1}\pa^\beta_{\Theta_2}\teE^\ve(\Theta_1,\Theta_2)_{\Theta_1=\Theta_2=\Theta}=\bsE\bigl(\;\pa^\alpha_\Theta\bsU_\ve(\Theta)\cdot\pa^\beta_\Theta\bsU_\ve(\Theta)\;\bigr). 
\]
Using the fact that $V^\ve$ is an even function we deduce  that for  any multi-indices $\alpha,\beta$ we have
\begin{subequations}
\begin{equation}
\pa^\alpha_{\vec{\theta}_1}\pa^\beta_{\vec{\theta}_2}\teE^\ve(\Theta,\Theta)= (2\pi\ve)^{-m} \ve^{-|\alpha|-|\beta|} (-1)^{|\alpha|} V^\ve_{\alpha+\beta}(0) , 
\label{eq: cova}
\end{equation}
\begin{equation} \pa^\alpha_{\vec{\vfi}_1}\pa^\beta_{\vec{\vfi}_2}\teE^\ve(\Theta,\Theta)=(2\pi\ve)^{-m} \ve^{-|\alpha|-|\beta|} (-1)^{|\alpha|}  V^\ve_{\alpha+\beta}(0) ,
\label{eq: covb}
\end{equation}
\begin{equation}
\pa^\alpha_{\vec{\theta}_1}\pa^\beta_{\vec{\vfi}_2} \teE^\ve(\Theta,\Theta)=(2\pi\ve)^{-m} \ve^{-|\alpha|-|\beta}(-1)^{|\alpha|} V^\ve_{\alpha+\beta}(\,\tau^\ve(\Theta) \,) , 
\label{eq: covc}
\end{equation}
\begin{equation}
\pa^\alpha_{\vec{\vfi}_1}\pa^\beta_{\vec{\theta}_2}\teE^\ve(\Theta,\Theta) =(2\pi\ve)^{-m}\ve^{-|\alpha|-|\beta|}(-1)^{|\beta|} V^\ve_{\alpha+\beta}(\;\tau^\ve(\Theta)\;).
\label{eq: covd}
\end{equation}
\end{subequations}

\subsection{The covariance form of $d\bsU_\ve(\Theta)$. }

Denote by  $\widetilde{\bsS}_\ve(\Theta)$ the covariance form of the Gaussian vector $d\bsU_\ve(\theta,\vfi)=d\bu^\ve(\theta)+d\bu^\ve(\vfi)$, 
\[
\widetilde{\bsS}_\ve(\Theta)=\left[
\begin{array}{cc}
\bsE\bigl(\,\pa_i\bu^\ve(\vec{\theta}) \pa_j\bu^\ve(\vec{\theta})\,\bigr)  &   \bsE(\pa_i\bu^\ve(\vec{\theta}) \pa_j\bu^\ve(\vec{\vfi})\,\bigr) \\
&\\
 \bsE\bigl(\,\pa_i\bu^\ve(\vec{\theta}) \pa_j\bu^\ve(\vec{\vfi}) \,\bigr)&  \bsE\bigl(\,\pa_i\bu^\ve(\vec{\vfi}) \pa_j\bu^\ve(\vec{\vfi})\,\bigr)
 \end{array}
 \right].
 \]
 Let us observe that this form is degenerate along the diagona $\bsD\subset \bT^m\times\bT^m$.   Denote by $N(\bsU^\ve)$ the number  of critical points of $\bsU^\ve$ situated outside the diagonal. Then
 \[
 N(\bsU^\ve)= N(\bu^\ve)^2-N(\bu^\ve)
 \]
 so that 
 \[
  \mu^\ve_{(2)}=\bsE\bigl(\, N(\bsU^\ve)\,\bigr)
  \]
  and the Kac-Rice formula implies that
 \begin{equation}
  \mu^\ve_{(2)}= \int_{\bT^m\times \bT^m\setminus \bsD} \frac {1}{\sqrt{ \det 2\pi \widetilde{\bsS}_\ve(\Theta)}} \bsE\Bigl(\, \bigr|\det \Hess \bsU^\ve(\Theta)\bigr|\,\Bigl|\, d\bsU^\ve(\Theta)=0\,\Bigr)|d\Theta|.
 \label{eq: KR3}
 \end{equation}
 Set
\[
r_\ve:=\frac{1}{2}|\tau^\ve|^2=\frac{1}{2\ve^2}|\tau|^2.
\] 
For any $\eta\in\bR^m$  and any smooth function $F:\bR^m\to \bR$  we  denote by $\eH(F,\eta)$ the quadratic form on $\bR^m\oplus \bR^m=T^*_{\vec{\theta}}\bT^m\oplus  T^*_{\vec{\vfi}}\bT^m$      whose value on  $X_-\oplus X_+\in\bR^m\oplus \bR^m$ is given by
\[
\eH(F,\eta)(X_-\oplus X_+)=-\Bigl(\,\pa^2_{X_-} F(0)+\pa^2_{X_+}F(0) +2\pa^2_{X_-X_+} F(\eta)\,\Bigr).
\]
If we fix an orthonormal basis $\be_1,\dotsc,\be_m$ of $\bR^m$ we obtain an orthonormal basis of $\bR^m\oplus \bR^m$
\[
\underline{\bsf}=\bigl\{\bsf_1=\be_1\oplus 0,\dotsc, \bsf_m=\be_m\oplus 0,\,\bsf_{m+1}=0\oplus \be_1,\dotsc,\bsf_{2m}= 0\oplus \be_m\,\bigr\}.
\]
In the basis $\underline{\bsf}$  the quadratic form $\eH(F,\eta)$ can be identified with the symmetric matrix $\eH(F,\eta)$ defined in (\ref{hfeta}),
\[
\eH(F,\eta)= \left[
\begin{array}{cc}
-\bsH(F, 0) & -\bsH(F,\eta)\\
-\bsH(F,\eta) & -\bsH(F,0)
\end{array}
\right].
\]
Using (\ref{eq: cova}-\ref{eq: covd}) we deduce that  
\begin{equation}
(2\pi)^m\ve^{m+2} \widetilde{\bsS}_\ve(\Theta)=\eH(V^\ve,\tau^\ve).
\label{eq: covdU}
\end{equation}
There is one first  issue  we need to address, namely the nondegeneracy  of $\eH(-V^\ve,\eta)$.

\subsection{Some quantitative nondegeneracy results}\label{ss: 53} We begin with  a technical result whose proof can be found in Appendix \ref{s: a}.
\begin{lemma}
Let 
\begin{equation}
\alpha(x):=\min\bigl(\,  \sin^2(x/2),  \cos^2(x/2)\,\bigr),\;\;\forall x\in\bR.
\label{eq: alpha}
\end{equation}
Then for any $t\geq 0$ we have
\begin{subequations}
\begin{equation}
|f'(0)|- \bigl\vert\, f'(t^2/2)+t^2f''(t^2/2)\,\bigr\vert \geq 2 \int_{\bR^m} \alpha(tx_1)x_1^2  w(x)|dx| ,
\label{eq: technical}
\end{equation}
\begin{equation}
|f'(0)|-|f'(t^2/2)| \geq 2 \int_{\bR^m} \alpha(tx_1)x_2^2  w(x)|dx|.
\label{eq: technicalb}
\end{equation}
\end{subequations}\qed
\label{lemma: tech}
\end{lemma}
\begin{lemma} The quadratic form  $\eH(V,\eta)$ is  nondegenerate for any $\eta\in\bR^m\setminus 0$. 
\end{lemma}

\begin{proof} Choose an orthonormal frame $(\be_1,\dotsc,\be_m)$ of $\bR^m$ such that $\eta=|\eta|\be_1$.  Using (\ref{eq: d2}) we deduce that
\begin{equation}
\begin{split}
\bsH(V, \eta)=f'(|\eta|^2/2)\one_m+ \diag\Bigl(\,|\eta|^2 f''(|\eta|^2/2),0,\dotsc, 0\,\Bigr)\\
=\diag\Bigl( \, f'\bigl(\, |\eta|^2/2\,\bigr)+|\eta|^2 f''\bigl(\, |\eta|^2/2\,\bigr),\, f'\bigl(\, |\eta|^2/2\,\bigr),\dotsc, f'\bigl(\,|\eta|^2/2\,\bigr)\,\Bigr).
\end{split}
\label{eq: hesv}
\end{equation}
Moreover, according to  (\ref{eq: technical}, \ref{eq: technicalb}) we have
\[
\vert\, f'(|\eta|^2/2)\,\vert,\;\;\bigl\vert\, f '(|\eta|^2/2)+|\eta|^2f''(|\eta|^2/2)\,\bigr\vert <\vert\,f'(0)\vert,\;\; \forall \eta\neq 0.
\]
We deduce
\[
\bsH(V, \eta)^2 < \bsH(V,0)^2, \;\;\forall \eta\neq 0.
\]
In particular $\bsH(V, \eta)^2 - \bsH(V,0)^2$ is invertible for any $\eta\neq 0$. We  set
\[
r=r(\eta):=|\eta|^2/2.
\]
Observe that if we let $\eta$ go to zero  along the line spanned by $\be_1$, then
\begin{equation}
\lim_{r\to 0} \frac{1}{r}\Bigl(\, \bsH(V,0)-\bsH(V,\eta)\,\Bigr)=-f''(0)\underbrace{\diag\bigl(\, 3, 1,\dotsc, 1,\bigr)}_{=:\dot{H}},
\label{eq: sing-lim}
\end{equation}
whereas
\[
\lim_{r\to 0} \Bigl(\, \bsH(V,0)+\bsH(V,\eta)\,\Bigr)=2 f'(0)\one_m.
\]
Hence
\[
\lim_{r\to 0} \frac{1}{r}\Bigl(\, \bsH(V,0)^2-\bsH(V,\eta)^2\,\Bigr)=-2f'(0)f''(0)\dot{H}.
\]
Let us point out that the notation $\dot{H}$ is a bit misleading. The symmetric  operator $\dot{H}$ depends on the unit vector $\frac{1}{|\eta|} \eta$, so it is really  a degree zero homogeneous map
\[
\dot{H}: \bR^m\setminus 0 \to \Sym_m,\;\; \eta \mapsto  \dot{H}(\eta)
\]
described explicitly by the equality
\[
\dot{H}(\eta)=-\Bigl(\, \one_m+2 P_{\eta}\,\bigr),
\]
where $P_{\eta}$ denote the orthogonal projection onto the line spanned by the  vector $\eta$.

Set
\[
\widetilde{\bsT}(\eta) = \left[
\begin{array}{cc}
-\bsH(V,0)  & \bsH(V,\eta)\\
 &\\
\bsH(V,\eta) & -\bsH(V, 0)
\end{array}
\right].
\]
Observe that
\[
\widetilde{\bsT}(\eta)\eH(V,\eta)= \left[
\begin{array}{cc} 
\bsH(V,0)^2-\bsH(V, \eta)^2 & 0\\
& \\
0 &  \bsH(V,0)^2-\bsH(V, \eta)^2
\end{array}
\right].
\]
The inverse of  $ \bsH(V,0)^2-\bsH(V,\eta)^2$, denoted by $R(\eta)$, is
\begin{equation} 
R(\eta)=\diag\left(\frac{1}{f'(0)^2-\bigl(f'(r)+2 r f''(r)\bigr)^2},\frac{1}{ f'(0)^2-f'(r)^2},\dotsc, \frac{1}{f'(0)^2-f'(r)^2}\,\right),
\label{eq: rinv}
\end{equation}
then  we deduce
\begin{equation}
\eH(V,\eta)^{-1}=R(\eta)T(\eta) =\left[
\begin{array}{cc}
-R(\eta)\bsH(V, 0)  & R(\eta)\bsH(V,\eta)\\
 &\\
R(\eta)\bsH(V,\eta) & -R(\eta)\bsH(V,0)
\end{array}
\right].
\label{eq: sinv}
\end{equation}
\end{proof} 

\begin{remark}  The above proof shows that there exists a constant $C>0$ such that
\begin{equation}
\left\| \eH(V,\eta)^{-1}\right\| \leq \underbrace{C\left( \frac{1}{ f'(0)^2- (f'(r)+|\eta|^2f''(r) )^2} +\frac{1}{f'(0)^2-f'(r)^2}\right)}_{=:\mu(\eta)},\;\;\forall \eta\in\bR^m\setminus 0.
\label{eq: lower}
\end{equation}
Observe that Using formula (\ref{eq: hesv})  where $r=\frac{1}{2}|\eta|^2$ we deduce after an elementary computation
\begin{equation}
\begin{split}
\det \eH(V, \eta)= & \det \left[\begin{array}{cc}
f'(0) &  f'(r)+|\eta|^2 f''(r)\\
f'(r)+|\eta|^2f''(r) &f'(0)
\end{array}
\right] \cdot\left( \det \left[\begin{array}{cc}
f'(0) &  f'(r)\\
f'(r) &f'(0)
\end{array}
\right]\right)^{m-1}\\
&=\left(\;f'(0)^2- \Bigr(|\eta|^2f''(r)-f'(r)\Bigr)^2 \right)\left(\; f'(0)^{2}-f'(r)^2\;\right)^{m-1}\\
&\sim 3\bigl(\;-f'(0)f''(0)\;\bigr)^m|\eta|^{2m}\;\;\mbox{as $\eta\to 0$}.
\end{split}
\label{det_explode}
\end{equation}
\qed
\end{remark}
We set
\[
G^\ve(\vec{\theta}):=V^\ve(\vec{\theta})-V^\ve(0)=\sum_{\vec{\nu}\in\bZ^m\setminus 0}V\left(\vec{\theta}+\frac{1}{\ve}\vec{\nu}\right).
\]
Observe that $G^\ve(\vec{\theta})$ is an even function so that, for any multi-index $\alpha$ such that $|\alpha|$ is even, the function $\pa^\alpha_\xi G^\ve$ is also even. Hence, under these circumstances, 
\[
\pa^\alpha_\xi G^\ve(\theta)-\pa^\alpha_\xi G^\ve(0)=O(|\vec{\theta}|^2)\;\;\mbox{as $\theta\to 0$}.
\]
We can  say a bit more. 

\begin{lemma}   Let $k,N\in\bZ_{>0}$. Then there exists a constant  $C=C_{k,N}>0$  and $\ve_0=\ve_0(k,N)>0$ such that for any multi-index $\alpha$, $i=0,1$,$|\alpha|=2k+i$, any $|\vec{\theta}|\leq 1$, and any $\ve<\ve_0$  we have
\[
\Bigl|\, \pa^\alpha_\xi G^\ve(\theta)-\pa^\alpha_\xi G^\ve(0)\,\Bigr|\leq C\ve^N |\vec{\theta}|^{2-i}.
\]
\label{lemma: dif-est}
\end{lemma}

\begin{proof} We consider only the case $i=0$. The case $i=1$ is dealt with in an analogous fashion. Set $V_\alpha:=\pa^\alpha_\xi V$. Note that  $V_\alpha$ is an even function and
\[
\pa^\alpha_\xi G^\ve(\theta)-\pa^\alpha_\xi G^\ve(0)=\sum_{\vec{\nu}\in\bZ^m\setminus 0} \left(\,V_\alpha\Bigr(\vec{\theta}+\frac{1}{\ve}\vec{\nu}\,\Bigr)-V_\alpha\Bigl( \frac{1}{\ve}\vec{\nu}\Bigr)\,\right)
\]
\[
=\frac{1}{2}\sum_{\vec{\nu}\in\bZ^m\setminus 0} \left(\,V_\alpha\Bigr(\vec{\theta}+\frac{1}{\ve}\vec{\nu}\,\Bigr)+V_\alpha\Bigr(-\vec{\theta}+\frac{1}{\ve}\vec{\nu}\,\Bigr)-2V_\alpha\Bigl( \frac{1}{\ve}\vec{\nu}\Bigr)\,\right)
\]
Now observe that the mean value theorem  implies that for $\vec{\nu}\in\bZ^m\setminus 0$ we have
\[
\left|\, V_\alpha\Bigr(\vec{\theta}+\frac{1}{\ve}\vec{\nu}\,\Bigr)+V_\alpha\Bigr(-\vec{\theta}+\frac{1}{\ve}\vec{\nu}\,\Bigr)-2V_\alpha\Bigl( \frac{1}{\ve}\vec{\nu}\Bigr)\right| \leq  \left(\sup_{|\eta-\ve^{-1}\vec{\nu}|\leq 2} | D^2 V_\alpha (\eta)|\right)|\vec{\theta}|^2
\]
(use the fact that $V\in\eS(\bR^m)$)
\[
\leq A_{N,k}\frac{\ve^N}{|\vec{\nu}|^N}
\]
for  $\ve>0$ sufficiently small. 
Hence
\[
\left|\pa^\alpha_\xi G^\ve(\theta)-\pa^\alpha_\xi G^\ve(0)\right|\leq\frac{A_{N,k}\ve^N}{2}\sum_{\vec{\nu}\in\bZ^m\setminus 0} \frac{1}{|\vec{\nu}|^N}.
\]
The above series is convergent as soon as $N>m$.
\end{proof}

\begin{lemma}  There exists $\ve_0>0$ such that  the following hold.

\smallskip

\noindent (a) For any $\ve<\ve_0$ and any $\eta\in\bR^m\setminus 0$ the operator $\eH(V^\ve,\eta)$ is invertible. 

\smallskip

\noindent (b) There exists   a constant $C>1$ such that for $\ve <\ve_0$ and $|\eta|<1$ we have 
\begin{equation}
\frac{1}{C}|\eta|^{2m} \leq \det \eH(V^\ve, \eta)\leq C |\eta|^{2m}.
\label{eq: det}
\end{equation}

\end{lemma}

\begin{proof} Let us observe that if $X=X(A,B)$ is a symmetric $2m\times 2m$ symmetric matrix  
\[
X=\left[
\begin{array}{cc}
A & B\\
B & A
\end{array}
\right]
\]
where $A$, $B$ are  symmetric $m\times m$ matrices,  and if the matrices $A$  and $A-BA^{-1}B$ are invertible, then $X$ is invertible and 
\begin{equation}
X^{-1}=\left[
\begin{array}{cc}
\bigl(\, A- BA^{-1}B\,\bigr)^{-1} & -A^{-1}B\bigl(\, A- BA^{-1}B\,\bigr)^{-1}\\
&
\\
-A^{-1}B\bigl(\, A- BA^{-1}B\,\bigr)^{-1} &  \bigl(\, A- BA^{-1}B\,\bigr)^{-1}
\end{array}
\right].
\label{eq: inv_block}
\end{equation}
The matrix $\eH(V^\ve,\eta)$ has this form $X(A,B)$ where
\[
A=A_\ve= \bsH(-V^\ve,0),\;\;B=B_\ve=B_\ve(\eta)=\bsH(-V^\ve,\eta).
\]
Again, we assume  that we have chosen an orthonormal  basis $\be_1,\dotsc,\be_m$ such that
\[
\eta=|\eta|\be_1.
\]
Since $V\in\eS(\bR^m)$ we deduce that
\[
A_\ve= -\bsH(V,0)+ O(\ve^N)
\]
so that there exists $\ve_1>0$ such that  $A_\ve$ is invertible for $\ve<\ve_1$.  Moreover
\[
A_\ve^{-1}=-\bsH(V,0)^{-1}\left(\,\one_m+O\bigl(\,\ve^N\,\bigr)\,\right)\stackrel{(\ref{eq: hesv})}{=}-\frac{1}{f'(0)}\left(\,\one_m+O\bigl(\,\ve^N\,\bigr)\,\right).
\]
Note that
\begin{equation}
B_\ve-A_\ve =\bsH(V^\ve,0)-\bsH(V^\ve,\eta)= \bsH(V,0)-\bsH(V,\eta) + \bsH(G_\ve,0)-\bsH(G_\ve,\eta) .
\label{eq: diff0}
\end{equation}
Using  (\ref{eq: sing-lim}) we deduce
\[
\bsH(V,0)-\bsH(V,\eta)=-rf''(0)\dot{H} + O(r^2),\;\;r=\frac{1}{2}|\eta|^2,
\]
while Lemma  \ref{lemma: dif-est} implies  that
\[
\bsH(G_\ve,0)-\bsH(G_\ve,\eta) = O(\ve^N r),
\]
where above, and in what follows,  the constants implied by the above $O$-symbols are independent of $\eta$ and $\ve$. Hence
\[
B_\ve-A_\ve=\bsH(V^\ve,0)-\bsH(V^\ve,\eta)= -f''(0)r\dot{H} + O(r\ve^N+r^2).
\]
Thus
\[
B_\ve A_\ve^{-1} =\left(\, A_\ve -rf''(0) \dot{H} + O(r\ve^N+r^2)\,\right) A_\ve^{-1}
\]
\[
= \one_m -  rf''(0)\dot{H} A_\ve^{-1}+O(r\ve^N+r^2)
\]
\[
=\one_m -  r\frac{f''(0)}{f'(0)}\dot{H} A_\ve^{-1}+O(r\ve^N+r^2),
\]
\[
B_\ve A_\ve^{-1}B_\ve = B_\ve A_\ve^{-1} \left(\, A_\ve -f''(0) r\dot{H} + O(r\ve^N+r^2)\,\right)
\]
\[
= A_\ve- 2f''(0)r\dot{H} + O(\ve^Nr+ r^2).
\]
Hence
\[
A_\ve -A_\ve B_\ve^{-1}A_\ve= 2f''(0) r\dot{H} + O(r\ve^N+r^2),
\]
We deduce that  there exists $\rho_0,\ve_2>0$ such that if $|\eta|<\rho_0$ and $\ve<\ve_2$,  then  $A_\ve-B_\ve A_\ve^{-1}$ is invertible and 
\[
\bigl(\,A_\ve-B_\ve A_\ve^{-1}\,\bigr)^{-1}=\frac{1}{2f''(0)r}\dot{H}^{-1}\Bigl(\,\one_m +O\bigl(\ve^N+r\,\bigr)\,\Bigr).
\]
We deduce that if $0<|\eta|<\rho_0$,  and $\ve <\min(\ve_1,\ve_2)$,  then the matrix $\eH(V^\ve,\eta)$ is invertible.  

Set
\[
\mu^*(\rho_0)=\sup_{|\eta|\geq \rho_0}\mu(\eta)
\]
where $\mu$ is defined in (\ref{eq: lower}). Since $\mu^*(\rho_0) <\infty$ and $\eH(V^\ve,\eta)$ converges uniformly to $\eH(V,\eta)$ on $|\eta|\geq \rho_0$ as $\ve\to 0$ we deduce from (\ref{eq: lower}) that  there exists $\ve_0<\min(\ve_1,\ve_2)$ such that $\eH(V^\ve,\eta)$ is invertible for $|\eta|\geq \rho_0$.

To prove (\ref{eq: det}) observe that 
\[
\eH(V^\ve,\eta)= \left[
\begin{array}{cc}
\one & B_\ve(\eta)A_\ve^{-1}\\
 B_\ve(\eta) A_\ve^{-1}& \one
\end{array}
\right] \cdot \left[
\begin{array}{cc}
A_\ve & 0\\
0 & A_\ve
\end{array}
\right].
\]
Set 
\[
C_\ve(\eta):=B_\ve(\eta)A_\ve^{-1}-\one.
\]
Then
\[
\det \left[
\begin{array}{cc}
\one & \one+ C_\ve(\eta)\\
\one + C_\ve(\eta) & \one
\end{array}
\right] 
= \det \left[
\begin{array}{cc}
\one &  C_\ve(\eta)\\
\one + C_\ve(\eta) & -C_\ve(\eta)
\end{array}
\right]= \det \left[
\begin{array}{cc}
2\one +C_\ve\eta &  0\\
\one + C_\ve(\eta) & -C_\ve(\eta)
\end{array}
\right]
\]
\[
= (-1)^m\det\bigl(\; 2\one-C_\ve(\eta)\;\bigr)\det C_\ve(\eta) .
\]
On the other hand
\[
C_\ve(\eta)= O(r)=O(|\eta|^2)
\]
so that
\begin{equation}
\det \eH(V^\ve,\eta)= O(|\eta|^{2m}).
\label{eq: det1}
\end{equation}
Thus, all the partial derivatives   at $0$ of order $< 2m$ of the function $\eta\mapsto \det \eH(V^\ve,\eta)$   are zero. Now observe that the family of functions $\bR^m\ni \eta \mapsto \det \eH(V^\ve, \eta)$ converges as $\ve\to 0$ in the the topology of $C^\infty(\bR^m)$ to  the function 
\[
\bR^m\ni \eta \mapsto \det \eH(V, \eta).
\]
The estimate (\ref{eq: det}) now follows from (\ref{det_explode})  coupled with  (\ref{eq: det1}).
\end{proof}

The above arguments, coupled with (\ref{eq: vve_appr}) prove a bit more, namely 
\begin{equation}
\sup_{|\eta|<1}|\eta|^{2m}\left|\,\frac{1}{\det \eH(V^\ve,\eta)} -\frac{1}{\det\eH(V,\eta)}\,\right|=O(\ve^N)\;\;\mbox{as $\ve\searrow 0$},\;\;\forall  N>0.
\label{eq: det_appr}
\end{equation}

\subsection{The behavior near the diagonal of the covariance form of $d\bsU^\ve$.}\label{ss: 54} Let 
\begin{equation}
\bsJ_m:=\bigl\{ -m,\dotsc, -1, 1,\dotsc, m\,\bigr\},\;\;\bsJ^\pm_m=\bigl\{\, i\in\bsJ_m;\;\;\pm i>0\,\bigr\}.
\label{jm}
\end{equation}
For any orthonormal basis  $\be_1,\dotsc,\be_m$ of $\bR^m$ set
\[
\tilde{\be}_{-i}=\be_i\oplus 0,\;\;\tilde{\be}_i=0\oplus \be_i,\;\;1\leq i\leq m.
\]
The collection $(\tilde{\be}_i)_{i\in\bsJ_m}$ is an orthonormal basis of $\bR^m\oplus \bR^m$. For $\eta\neq 0$ we denote by $\tsi_{i,j}^\ve=\tsi^\ve_{i,j}(\eta)$, $i,h\in\bsJ_m$  the entries of the  matrix  $\eH(V^\ve,\eta)^{-1}$ with respect to the basis $(\tilde{\be}_i)_{i\in\bsJ_m}$. These entries  satisfy the symmetry conditions
\begin{equation}
\tsi^\ve_{i,j}=\tsi^\ve_{-i,-j},\;\;\forall i,j\in \bsJ_m.
\label{eq: sinv1}
\end{equation}
Similarly, we denote by $\tsi_{i,j}^0=\tsi^0_{i,j}(\eta)$, $i,j\in\bsJ_m$  the entries of the  matrix  $\eH(V,\eta)^{-1}$ with respect to the basis $(\tilde{\be}_i)_{i\in\bsJ_m}$.  

The equality (\ref{eq: inv_block}) implies that  for $i,j>0$    and $\ve\geq 0$  we have the  equalities of matrices
\begin{subequations}
\begin{equation}
\bigl(\tsi^\ve_{i,j}\bigr)_{1\leq i,j\leq m}= \Bigl( A_\ve -B_\ve(\eta)A_\ve^{-1} B_\ve(\eta) \;\Bigr)^{-1},
\label{eq: sia}
\end{equation}
\begin{equation} \bigl(\tsi^\ve_{i,-j}\bigr)_{1\leq i,j\leq m}= -A_\ve^{-1}B_\ve(\eta) \Bigl( A_\ve -B_\ve(\eta)A_\ve^{-1} B_\ve(\eta)\;\Bigr)^{-1}.
\label{eq: sib}
\end{equation}
\end{subequations}

\begin{lemma} Fix  an orthonormal basis  $\be_1,\dotsc,\be_m$ of $\bR^m$. Then for   $\ve>0$ sufficiently small the matrix 
\[
\Bigl(\;V_{i,j,1,1}^\ve(0)\;\Bigr)_{1\leq i,j\leq m}=\Bigl(\;V_{1,1,i,j}^\ve(0)\;\Bigr)_{1\leq i,j\leq m}
\]
is invertible. Moreover
\begin{subequations}
\begin{equation}
\lim_{t\to 0} t^2\Bigl(\;\tsi_{i,j}^\ve(t\be_1)\;\Bigr)_{1\leq i,j\leq m}= \Bigl(\; V_{i,j,1,1}^\ve(0)\;\Bigr)_{1\leq i,j\leq m}^{-1},
\label{eq: silima}
\end{equation}
\begin{equation}
\lim_{t\to 0} t^2\Bigl(\;\tsi_{-i,j}^\ve(t\be_1)\;\Bigr)_{1\leq i,j\leq m}= -\Bigl(\; V_{i,j,1,1}^\ve(0)\;\Bigr)_{1\leq i,j\leq m}^{-1},
\label{eq: silimb}
\end{equation}
\end{subequations}
\underline{uniformly} for    sufficiently small positive $\ve$.
\label{lemma: silim}
\end{lemma}

\begin{proof} Observe that for any $N>0$ we have
\[
V_{i,j,1,1}^\ve(0)=V_{i,j,1,1}(0)+O(\ve^N),
\]
and
\[
\Bigl(\; V_{i,j,1,1}(0)\;\Bigr)_{1\leq i,j\leq m}\stackrel{(\ref{eq: d4})}{ =} \diag(3,1,\dotsc, 1)
\]
This proves that  the matrix
\[
\Bigl(\;V_{i,j,1,1}^\ve(0)\;\Bigr)_{1\leq i,j\leq m}
\]
is invertible for $\ve>0$ sufficiently small. Observe that
\[
\Bigl(\tsi_{i,j}^\ve(\eta)\;\Bigr)_{1\leq i,j\leq m}= \Bigl( A_\ve- B_{\ve}(\eta)A_{\ve}^{-1} B_\ve(\eta)\;\Bigr)^{-1},
\]
\[
B_{\ve}(\eta)=\bsH(-V^\ve(\eta)\,)= \Bigl(\;-V_{i,j}^\ve(\eta)\;\Bigr)_{1\leq i ,j\leq m},\;\;A_\ve= B_\ve(0),
\]
\[
B_\ve(t\be_1)  =B_\ve(0)+\frac{t^2}{2}\ddot{B}_\ve+O(t^4),\;\;\ddot{B}_\ve = \Bigl(\; -V_{i,j,1,1}^\ve(0)\;\Bigr)_{1\leq i ,j\leq m}.
\]
Above and  in  what follows, the constants implied by the $O$-symbol are \emph{independent} of $\ve$ sufficiently small. Hence
\[
B_\ve(t\be_1)B_\ve(0)^{-1} B_\ve(t\be_1)=  B_\ve(0)+ t^2\ddot{B}_\ve+ O(t^4),
\]
\[
B_\ve(0) -B_\ve(t\be_1)B_\ve(0)^{-1} B_\ve(t\be_1)= -t^2\ddot{B}_\ve+O(t^4),
\]
\[
\Bigl( A_\ve- B_{\ve}(t\be_1)A_{\ve}^{-1} B_\ve(t\be_1)\;\Bigr)^{-1}= -t^{-2}\ddot{B}_\ve +O(t^4),
\]
\begin{equation}
t^2\Bigl(\; \tsi_{ij}^\ve(t\be_1)\;\Bigr)_{1\leq i,j\leq m} =-\ddot{B}_\ve^{-1}+ O(t^2).
\label{eq: silimc}
\end{equation}
This proves (\ref{eq: silima}), including the uniform convergence. The equality (\ref{eq: silimb}) is proved in a similar fashion.
\end{proof}

 Denote by $R_\ve(t)$ the error in the approximation (\ref{eq: silimc}), i.e.,
\[
R_\ve(t)=t^2\Bigl(\; \tsi_{i,j}^\ve(t\be_1)\;\Bigr)_{1\leq i,j\leq m} +\ddot{B}_\ve^{-1}.
\]
A quick look at  the proof of  the above lemma shows that $R_\ve(t)$ converges to $R_0(t)$  as $\ve\searrow 0$ in the topology of $C^\infty(\, -t_0,t_0\,)$, where $t_0$ is some small positive number.   Moreover (\ref{eq: vve_appr}) implies that  
\[
\Vert R_\ve-R_0\Vert_{C^0(-t_0,t_0)}= O(\ve^N),\;\;\mbox{as $\ve\searrow 0$},\;\;\forall N\geq 0.
\]   
This observation has the following important consequence.

\begin{corollary}  For any $\ve>0$ sufficiently small the function
\[
\bR\setminus 0\ni t \mapsto t^2 \tsi^\ve_{i,j}(t\be_1)
\]
admits a \emph{smooth} extension to $\bR$.  Moreover, there exists $t_0>0$ such that for all $i,j\in\bsJ_m$ the smooth functions
\[
 (-t_0,t_0)\ni t \mapsto t^2\tsi^\ve_{i,j}(t\be_1)
 \]
 converge as $\ve\to 0$  in  the  topology of $C^\infty(-t_0,t_0)$ to the  smooth function
 \[
 (-t_0,t_0)\ni t\mapsto  t^2\tsi^0_{i,j}(t\be_1),
\]
and
\[
\sup_{|t|<t_0}\bigl\vert\, t^2\tsi^\ve_{i,j}-t^2\tsi^0_{i,j}\,\bigr\vert=O(\ve^N),\;\;\forall N\geq 0. \proofend
\]
\label{cor: silim}
\end{corollary}

We will denote by $K^\ve_{ij}$ the entries  of the \emph{inverse} of the matrix
\[
\Bigl(\;V_{1,1,i,j}^\ve(0)\;\Bigr)_{1\leq i,j\leq m},
\]
so that
\begin{equation}
\sum_{a>0} V^\ve_{1,1,j,a}K_{ab}=\delta_{jb},\;\;\forall 1\leq j,b\leq m.
\label{eq: inv}
\end{equation}

 \subsection{Conditional hessians} \label{ss: 55} Fix $\vec{\Theta}:=(\vec{\theta},\vec{\vfi})\in\bT^m\times \bT^m$  set $\tau= \vec{\vfi}-\vec{\theta}$ and fix an orthonormal basis  $(\be_1,\dotsc,\be_m)$ of $\bR^m$ such that 
\[
\tau=|\tau|\be_1.
\]
We obtain  a basis  $(\tilde{\be}_i)_{i\in\bsJ_m}$ of $T_{\vec{\Theta}}\bT^m\times \bT^m$. 

Using these conventions we can regard  $\widetilde{\bsS}_\ve(\Theta)$ as a matrix with  entries $s_{ij}^\ve$, $i,j\in\bsJ_m$. The entries of the the matrix 
\[
(2\pi)^{-m}\ve^{-m-2} \widetilde{\bsS}_\ve(\Theta)^{-1}
\]
 are $\bsi^\ve_{a,b}(\tau^\ve)$, $a,b\in\bsJ_m$. For $i_1,\dotsc, i_k\in\bsJ_m$ we set
\[
\bsU^\ve_{i_1,\dotsc, i_k}:=\pa^k_{\tilde{\be}_{i_1}\dotsc\tilde{\be}_{i_k}}\bsU_\ve(\vec{\Theta}),\;\;V^\ve_{i_1,\dotsc, i_k}(\eta): =V^\ve_{|i_1|,\dotsc ,|i_k|}(\eta).
\]
We have the random matrix
\[
\widetilde{\bsH}^\ve=\left( \,\bsU^\ve_{i,j}\,\right)_{1\leq |i|,|j|\leq m}\in\Sym_m^{\times 2},
\]
and the conditional random matrix
\[
\widehat{\bsH}^\ve := \Bigl(\, \widetilde{\bsH}^\ve\,\bigl| \, d\bsU_\ve=0\,\Bigr)\in\Sym_m^{\times 2}.
\]
Both random matrices $\reh$ and $\coh$ admit block decompositions
\[
\reh=\reh_-\oplus \reh_+,\;\;\coh=\coh_-\oplus \coh_+
\]
corresponding to the partition $\bsJ_m=\bsJ_m^-\sqcup\bsJ_m^+$. We denote by $\widehat{\bsU}^\ve_{i,j}$ the entries of $\coh$ and we set
\[
\widetilde{\Omega}^\ve_{i,j|k,\ell}(\Theta) :=\bsE\bigl(\bsU^\ve_{i,j}\bsU^\ve_{k,\ell}),\;\; i,j,k,\ell \in\bsJ_m,
\]
\[
\widehat{\Xi}^\ve_{i,j|k\ell}(\Theta) :=\bsE\bigl(\,\widehat{\bsU}^\ve_{i,j}\widehat{\bsU}^\ve_{k,\ell}),\;\; i,j,k,\ell \in\bsJ_m.
\]
Using the regression formula \cite[Prop. 1.2]{AzWs}  we deduce 
\[
\widehat{\Xi}^\ve_{i,j|k,\ell}(\Theta)=\widetilde{\Omega}^\ve_{i,j|k,\ell}(\Theta)-(2\pi)^{m}\ve^{m+2}\sum_{a,b\in\bsJ_m} \bsE\bigl(\, \bsU^\ve_{i,j}\bsU^\ve_a\,\bigr)\tsi^\ve_{ab} \bsE\bigl(\,\bsU^\ve_b\bsU^\ve_{k,\ell}\,\bigr).
\]
Observe that
\begin{subequations}
\begin{equation}
\bsU^\ve_{i,j}=0,\;\;\mbox{if $i\cdot j<0$},
\label{eq: U1}
\end{equation}
\begin{equation}
\tom^\ve_{i,j|k,\ell}(\Theta)= 0,\;\;\mbox{if $i\cdot j<0$ or $k\cdot \ell<0$},
\label{eq: tom1}
\end{equation}
\begin{equation}
\tom^\ve_{i,j|k,\ell}(\Theta)=\tom^\ve_{-i,-j|-k-\ell}(\Theta),\;\;\forall i,j,k,\ell\in\bsJ_m.
\label{eq: tom2}
\end{equation}
\end{subequations}
Using (\ref{eq: covc}) we deduce that if $a,i,j>0$, then
\begin{subequations}
\begin{equation}
\bsE\bigl(\,\bsU^\ve_{i,j}\cdot \bsU^\ve_a\,\bigr)= \bsE\bigl(\,\bsU_{-i,-j}^\ve\cdot \bsU^\ve_{-a}\,\bigr)= (2\pi)^{-m}\ve^{-m-3} V^\ve_{i,j,a}(0),
\label{eq: 3cov0a}
\end{equation}
\begin{equation}
\bsE\bigl(\,\bsU^\ve_{-i,-j}\cdot \bsU^\ve_a\,\bigr)= -\bsE\bigl(\,\bsU_{i,j}^\ve\cdot \bsU^\ve_{-a}\,\bigr)= (2\pi)^{-m}\ve^{-m-3} V^\ve_{i,j,a}(\tau^\ve).
\label{eq: 3cov0b}
\end{equation}
\end{subequations}
Using (\ref{eq: covb})  and the fact that $V^\ve$ is an even function we deduce that  if $a,i,j>0$, then
\begin{equation}
\bsE\bigl(\, \bsU^\ve_{-i,-j}\bsU^\ve_{-a}\,\bigr)= \bsE\bigl(\, \bsU^\ve_{i,j}\bsU^\ve_{a}\,\bigr) = 0.
\label{eq: 3cov}
\end{equation}
Invoking (\ref{eq: cova}, \ref{eq: covb}) we deduce that if $i,j,k,\ell>0$, then
\[
\widetilde{\Omega}^\ve_{i,j|k,\ell}(\Theta)=  \widetilde{\Omega}^\ve_{-i,-j|-k,-\ell}= (2\pi)^{-m}\ve^{-4-m}V^\ve_{i,j,k,\ell}(0),
\]
while (\ref{eq: covc}, \ref{eq: covd}) imply that
\[
\widetilde{\Omega}^\ve_{-i,-j|k,\ell}(\Theta)=  \widetilde{\Omega}^\ve_{i,j|-k,-\ell}= (2\pi)^{-m}\ve^{-4-m}V^\ve_{i,j,k,\ell}(\tau^\ve).
\]
Hence if $i,j,k,\ell>0$ then
\begin{subequations}
\begin{equation}
(2\pi)^m\ve^{m+4} \widehat{\Xi}^\ve_{i,j|k,\ell}(\Theta)=V^\ve_{i,j,k,\ell}(0) -\sum_{a,b>0}V^\ve_{i,j,a}(\tau^\ve)V^\ve_{k,\ell, b}(\tau^\ve)\tsi^\ve_{a,b},
\label{eq: condhessa}
\end{equation}
\begin{equation}
(2\pi)^m\ve^{m+4} \widehat{\Xi}^\ve_{-i,-j|k,\ell}(\Theta)=V^\ve_{i,j,k,\ell}(\tau^\ve) +\sum_{a,b>0}V^\ve_{i,j,a}(\tau^\ve)V^\ve_{k,\ell, b}(\tau^\ve)\tsi^\ve_{a,-b},
\label{eq: condhessb}
\end{equation}
\begin{equation}
\widehat{\Xi}^\ve_{-i,-j|-k,-\ell}(\Theta)=\widehat{\Xi}^\ve_{i,j|k,\ell}(\Theta).
\label{eq: cndhessc}
\end{equation}
\end{subequations}
For $\eta\in\bR^m\setminus 0$   and $i,j,k,\ell>0$ we set
\begin{subequations}
\begin{equation}
\bbom^\ve_{i,j|k,\ell}(\eta)=\bbom^\ve_{-i,-j|-k,-\ell}(\eta):= V^\ve_{i,j,k,\ell}(0) -\sum_{a,b>0}V^\ve_{i,j,a}(\eta)V^\ve_{k,\ell, b}(\eta)\tsi^\ve_{a,b}(\eta),
\label{eq: bboma}
\end{equation}
\begin{equation}
\bbom^\ve_{-i,-j|k,\ell}(\eta):=V^\ve_{i,j,k,\ell}(\eta) +\sum_{a,b>0}V^\ve_{i,j,a}(\eta)V^\ve_{k,\ell, b}(\eta)\tsi^\ve_{a,-b}(\eta).
\label{eq: bbomb}
\end{equation}
\end{subequations}
The collection
\begin{equation}\label{barbbom}
\bbom^\ve(\eta):=(\bbom^\ve_{i,j|k,\ell})_{i,j,k,\ell\in \bsJ_m},\;\;\eta \in\bR^m\setminus 0,
\end{equation}
 describes the covariance forms of a gaussian measure $\Gamma_{\bbom^\ve(\eta)}$  on the space $\Sym_m^{\times 2}$.  By construction
 \begin{equation}
 \widehat{\Xi}^\ve(\Theta)=\frac{1}{(2\pi)^m\ve^{m+4} } \bbom^\ve\bigl(\, \tau^\ve(\Theta)\,\bigr).
 \label{eq: hat-bar}
 \end{equation}
For any $\eta\neq 0$ the Gaussian measure $\Gamma_{\bbom^\ve(\eta)}$   on $\Sym_m^{\times 2}$  describes a random matrix 
\[
B=B^+\oplus B^+, \;\;B^\pm\in \Sym_m 
\]
characterized as follows.

\begin{itemize}

\item The two components $B^\pm$ are identically distributed Gaussian random symmetric matrices.  

\item The covariance form  of the distribution of $B^+$ is given by $(\bbom^\ve(\eta)_{i,j|k,\ell})_{i,j,k,\ell\in\bsJ_m^+}$.    

\item The correlations   between the two components  $B^\pm$ are described by $(\,\bbom^\ve_{-i,-j|k,\ell}(\eta)\,)_{i,j,k,\ell\in\bsJ_m^+}$. 

\end{itemize}

The Gaussian  measure on $\Sym_m$ defined by  $(\bbom^\ve(\eta)_{i,j|k,\ell})_{i,j,k,\ell\in\bsJ_m^+}$ is invariant under the subgroup $O_\eta(m)$ consisting of orthogonal transformations of $\bR^m$ that fix $\eta$. In Appendix \ref{s: gmat} we give a more detailed description of the $O_\eta(m)$-invariant Gaussian measures on $\Sym_m$.

\subsection{Putting all the parts together}\label{ss: 56} From the  Kac-Rice formula we deduce that
\[
\rho_2^\ve(\Theta)= \frac{1}{(2\pi)^m\sqrt{\det  \widetilde{\bsS}_\ve(\,\tau^\ve(\Theta)\,)}}\int_{\Sym_m^{\times 2}} |\det B| \, \Gamma_{\widehat{\Xi}^\ve(\, \Theta\,)}(|dB|).
\]
From (\ref{eq: covdU}) we deduce 
\[
\frac{1}{\sqrt{\det  \widetilde{\bsS}_\ve(\eta)}}= \frac{(2\pi)^{m^2}\ve^{m(m+2)}}{\sqrt{\det \eH(V^\ve,\eta)}}.
\]
If $B\in\Sym_m^{\times 2}$ is a random matrix  distributed accoding to $\Gamma_{\widehat{\Xi}^\ve(\Theta)}$, then the equality  (\ref{eq: hat-bar}) shows that the random matrix $C=(2\pi)^{\frac{m}{2}}\ve^{\frac{m+4}{2}} B$ is distributed according to  $\Gamma_{\bbom^\ve(\tau^\ve(\Theta))}$. Observing that
\[
|\det B|=\frac{1}{(2\pi)^{m^2} \ve^{m(m+4)}}|\det C|.
\]
Putting all the above together we obtain the following important formula
\begin{equation}
\rho^\ve_2(\Theta)=\frac{\ve^{-2m}}{(2\pi)^m\sqrt{\det \eH(V^\ve,\tau^\ve(\Theta)\,)}} \int_{\Sym_m^{\times 2}} |\det B| \, \Gamma_{\bbom^\ve(\tau^\ve(\Theta)\,)}(|dB|).
\label{eq: rho2}
\end{equation}

\section{The proof of Theorem \ref{th: main}}\label{s: 6}
\setcounter{equation}{0}

\subsection{The off-diagonal behavior of $\rho_2^\ve$}\label{ss: 61}  For $\Theta\in\bR^m\times \bR^m$  we  define $\tau(\Theta)$ to be the unique vector in $[0,1)^m\subset \bR^m$ such that
\[
\tau(\Theta)-(\vec{\vfi}-\vec{\theta})\in\bZ^m.
\]
We set
\[
\Vert\tau(\Theta)\Vert:=\dist\bigr(\, (\vec{\vfi}-\vec{\theta},\bZ^m\,\bigr).
\]
The function
\[
\bR^m\times\bR^m\ni \Theta\mapsto \Vert\tau(\Theta)\Vert\in [0,\infty)
\]
descends to a  continuous function
\[
\bT^m\times\bT^m\ni \Theta\mapsto \Vert\tau(\Theta)\Vert\in [0,\infty).
\]
For $\hbar>0$ sufficietly small consider the     region
\[
\eC_\hbar:=\bigl\{ \,\Theta\in\bT^m\times\bT^m;\;\;\Vert\tau(\Theta)\Vert\geq \hbar\,\bigr\}.
\]
For $\hbar$ small the above set  $\eC_\hbar$ is the complement of a small tubular neighborhood  of the diagonal $\bsD$.  For $i,j,k,\ell\in\bsJ_m^+$ we  set
\[
\bbom^\infty_{i,j|k,\ell} =\bbom^\infty_{-i,-j|-k,-\ell}:= V_{i,j,k,\ell}(0),
\]
\[
\bbom^\infty_{-i,-j|k,\ell} =\bbom^\infty_{i,-j|k,-\ell}:=0.
\]
The  collection $(\bbom^\infty_{i,j|k,\ell})_{i,j,k,\ell\in\bsJ_m^+}$ describes the covariance form of  an $O(m)$-invariant Gaussian measure  on $\Sym_m$.   As explained in Appendix \ref{s: gmat}, there exists a two-parameter  family  $\bGamma_{u,v}$ of such measures on   $\Sym_m$. The equalities (\ref{eq: deriv2}) show that the measure   defined by  $(\bbom^\infty_{i,j|k,\ell})_{i,j,k,\ell\in\bsJ_m^+}$ corresponds to
\[
u=v =f''(0)\stackrel{(\ref{eq: sdh1})}{=} h_m=\int_{\bR^m} x_1^2x_2^2 w(|x|) dx.
\]
The collection $(\bbom^\infty_{i,j|k,\ell})_{i,j,k,\ell\in\bsJ_m^+}$   describes the product Gaussian measure 
\[
\Gamma_{\bbom^\infty}=\bGamma_{h_m,h_m}\times \bGamma_{h_m,h_m}.
\]
Statistically,  $\Gamma_{\bbom^\infty}$  describes a pair an independent     random   symmetric $m\times m$ matrices  each distributed  according to $\bGamma_{h_m,h_m}$.  We set
\[
\tilde{\eC}_\hbar=\bigl\{ \tau\in\bR^m;\;\dist(\tau,\bZ^m)\geq \hbar\,\bigr\}.
\]
Since $V$ is a Schwartz  function,   we deduce from (\ref{eq: vve}) that the functions $\tau\mapsto V^\ve(\frac{1}{\ve}\tau)$  and their derivatives converge \emph{uniformly} on $\tilde{\eC}_\hbar$ to $0$.  More precisely for any $K>0$  and any $N>0$ there exists $C=C(K,N,\hbar)$ such that
\[
|\pa^\alpha V^\ve(\tau/\ve)|\leq C \ve^N,\;\;\forall \tau\in\tilde{\eC}_\hbar, \;\;|\alpha|\leq K.
\]
This implies that for any $N>0$ we have the following   estimate, \emph{uniform} on $\eC_\hbar$
\begin{equation}
\bigl\Vert \eH\bigl(V^\ve,\ve^{-1}\tau(\Theta)\,\bigr)- \eH_\infty(V)\,\bigr\Vert= O(\ve^N)\;\;\mbox{as $\ve\searrow 0$},
\label{eq: cov-offd}
\end{equation}
where $\eH_\infty$ was defined in (\ref{hinfty}). This implies that 
\begin{equation}
\bigl| \det \eH\bigl(V^\ve,\ve^{-1}\tau(\Theta)\,\bigr)- \det \eH_\infty(V)\,\bigr|= O(\ve^N)\;\;\mbox{as $\ve\searrow 0$},
\label{eq: cov-offd1}
\end{equation}
uniformly in $\Theta\in\eC_\hbar$. The equalities   (\ref{eq: bboma}) and (\ref{eq: bbomb}) that  $\forall i,j,k,\ell\in\bsJ_m$ show that
\begin{equation}
\Bigl| \bbom^\ve_{i,j|k,\ell}\bigl(\ve^{-1}\tau(\Theta)\,\bigr)-\bbom^\infty_{i,j|k.\ell}\Bigr|= O(\ve^N)\;\;\mbox{as $\ve\searrow 0$},
\label{eq: cov-offd2}
\end{equation}
uniformly in $\Theta\in\eC_\hbar$. Recalling that $\tau^\ve(\Theta)=\ve^{-1}\tau(\Theta)$ and  $\eH_\infty(V)$ is invertible,  we deduce from   Proposition \ref{prop: diff-gauss},  (\ref{eq: cov-offd1}) and (\ref{eq: cov-offd2}) that as $\ve\searrow 0$
\begin{equation}
\begin{split}
\frac{\ve^{-2m}}{(2\pi)^m\sqrt{\det \eH(V^\ve,\tau^\ve(\Theta)\,)}} \int_{\Sym_m^{\times 2}} |\det B| \, \Gamma_{\bbom^\ve(\tau^\ve(\Theta)\,)}(|dB|)\\
= \frac{\ve^{-2m}}{(2\pi)^m\sqrt{\det \eH_\infty(V)}} \int_{\Sym_m^{\times 2}} |\det B| \, \Gamma_{\bbom^\infty}(|dB|) +O(\ve^{N-2m}),
\end{split}
\label{eq: cov-offd3}
\end{equation}
uniformly in $\Theta\in\eC_\hbar$. Arguing as above, using (\ref{eq: upsivea}) instead of   (\ref{eq: bboma}) and (\ref{eq: bbomb}),    we  deduce in similar fashion that as $\ve\searrow 0$ we have
\begin{equation}
\begin{split}
\frac{\ve^{-2m}}{(2\pi)^m\sqrt{\det \eH_\infty(V^\ve)}} \int_{\Sym_m^{\times 2}} |\det B| \, \Gamma_{\Upsilon^\ve\times \Upsilon^\ve}(|dB|)\\
= \frac{\ve^{-2m}}{(2\pi)^m\sqrt{\det \eH_\infty(V)}} \int_{\Sym_m^{\times 2}} |\det B| \, \Gamma_{\bbom^\infty}(|dB|) +O(\ve^{N-2m}),
\end{split}
\label{eq: cov-offd4}
\end{equation}
uniformly in $\Theta\in\eC_\hbar$.

Using (\ref{eq: tildro1}) and (\ref{eq: rho2}) we deduce that for any $N>0$
\begin{equation}
\rho_2^\ve(\Theta)-\tilde{\rho}^\ve_1(\Theta)= O(\ve^{N-2m}), \;\;\mbox{as $\ve\searrow 0$},
\label{eq: dif-ro}
\end{equation}
uniformly in $\Theta\in\eC_\hbar$.

\subsection{The behavior of the conditional hessians near the diagonal}\label{ss: 62}
Observe that  both functions  $\bbom^\ve_{ij|k,\ell}(\eta)$ and $\bbom^\ve_{-i,-j|k,\ell}(\eta)$ are even and smooth on $\bR^m\setminus 0$. Moreover, they  restrict to  smooth functions on each one-dimensional subspace of $\bR^m$.  Indeed, if  as usual we  pick an orthonormal basis  $\be_1,\dotsc,\be_m$ of  $\bR^m$ such  $\eta=t\be_1$, then for $t\neq 0$ we have,  
\[
\sum_{a,b>0}\frac{1}{t}V^\ve_{i,j,a}(t\be_1)\frac{1}{t}V^\ve_{k,\ell, b}(t\be_1) t^2\tsi^\ve_{a,b}(t\be_1),
\]
and each of the functions
\[
t\mapsto \frac{1}{t}V^\ve_{i,j,a}(t\be_1),\;\;t\mapsto t^2\bsi^\ve_{a,b}(t\be_1)
\]
extends to smooth  functions on $\bR$ satisfying
\[
\lim_{t\to 0}\frac{1}{t}V^\ve_{i,j,a}(t\be_1)= V^\ve_{ija1}(0),\;\; \lim_{t\to 0}t^2\bsi^\ve_{a,b}(t\be_1)= K^\ve_{ij}.
\]
Moreover, $V^\ve\to V$ in the natural topology of $C^\infty(\bR^m)$  we deduce  from Corollary \ref{cor: silim}  the following important result.

\begin{lemma} Let $t_0>0$ be as in Corollary \ref{cor: silim}.  For any $i,j,k,\ell>0$ the smooth functions
\[
(0,t_0)\ni t\mapsto  \bbom^\ve_{i,j|k,\ell}(t\be_1),\;\;(0,t_0)\ni t\mapsto \bbom^\ve_{-i,-j|k,\ell}(t\be_1),
\]
converge in the topology of $C^\infty(0,t_0)$   as $\ve\searrow 0$  to the  smooth functions
\[
(0,t_0)\ni t\mapsto \bbom^0_{i,j|k,\ell}(t\be_1):= V_{i,j,k,\ell}(0) -\sum_{a,b>0}V_{i,j,a}(t\be_1)V_{k,\ell, b}(t\be_1)\tsi^0_{a,b}(t\be_1),
\]
and
\[
(0,t_0)\ni  t\mapsto \bbom^0_{-i,-j|k,\ell}(t\be_1):=V_{i,j,k,\ell}(t\be_1) +\sum_{a,b>0}V_{i,j,a}(t\be_1)V_{k,\ell, b}(t\be_1)\tsi^0_{a,-b}(t\be_1).\proofend
\] 
\label{lemma: xilim}
\end{lemma} 
A more explicit  description of the covariances  $(\bbom^0_{i,j|k,\ell})_{i,j,k,\ell\in\bsJ_m}$ can be found in  Section \ref{ss: 63}. The above  result has an immediate consequence.
\begin{corollary}  As $\ve \searrow 0$ we have
\[
\bbom^\ve_{i,j|k,\ell}(\eta)=O(1),\;\;\forall  i,j,k,\ell\in\bsJ_m,
\]
\emph{uniformly} in $|\eta|\leq 1$. \qed
\label{cor: O1}
\end{corollary}

The above estimate can be    substantially improved  in some instances.

\begin{lemma}  Assume $\eta=t\be_1$, $t>0$. For any $i,j,k,\ell\in \bsJ_m$ and any $0<\ve\ll 1$ we have
\begin{equation}
\bbom^\ve_{\pm1,j|k,\ell}(0)=0.
\label{eq: O3}
\end{equation}
Moroever, as $\ve \searrow 0$  the smooth functions
\begin{equation}
(0,t_0)\ni t\mapsto t^{-2}\bbom^\ve_{\pm 1,j|k,\ell}(t\be_1)
\label{eq: O2}
\end{equation}
converge \emph{uniformly}  on $(0,t_0)$ to the smooth functions
\[
(0,t_0)\ni t\mapsto t^{-2}\bbom^0_{\pm 1,j|k,\ell}(t\be_1).
\]
\label{lemma: O2}
\end{lemma}

\begin{proof}  It suffices to show that
\begin{equation}
\bbom^\ve_{1,j|k,\ell}(0)=0,\;\;\forall j,k,\ell\in \bsJ_m,\;\; 0<\ve \ll 1,
\label{eq: O3a}
\end{equation}
because  the    similar equalities
\[
\bbom^\ve_{-1,j|k,\ell}(0)=0,\;\;\forall j,k,\ell\in \bsJ_m,\;\; 0<\ve \ll 1
\]
follow  from  (\ref{eq: O3a}) via the symmetry conditions (\ref{eq: tom1}, \ref{eq: tom2}) and the defining equations (\ref{eq: condhessa}, \ref{eq: condhessb}).

Note that (\ref{eq: O3a}) holds if $j<0$ or $k\cdot \ell <0$. We assume $j>0$ and we distinguish two cases.

\smallskip

\noindent {\bf A.} $k,\ell>0$. Using  the equality
\[
\bbom^\ve_{1,j|k,\ell}(t\be_1)=V^\ve_{1,j,k,\ell}(0)-\sum_{a,b>0}\frac{1}{t}V^\ve_{1,j,a}(t\be_1)\frac{1}{t}V^\ve_{k,\ell, b}(t\be_1) t^2\tsi^\ve_{a,b}(t\be_1).
\]
Letting $t\to 0$ and invoking (\ref{eq: silima}) and (\ref{eq: inv}) we deduce
\[
\bbom_{1,j|k,\ell}^\ve(0)= V^\ve_{1,j,k,\ell}(0)-\sum_{b>0}\; \underbrace{\left(\sum_{a>0}V^\ve_{1,1,j,a}(0)K^\ve_{ab}\right)}_{=\delta_{jb}}V^\ve_{1,b,k,\ell}(0)=0.
\]
\noindent {\bf B.} $k,\ell<0$. Use  the equalities (\ref{eq: bbomb}), (\ref{eq: silimb}) , (\ref{eq: inv}) and argue as in {\bf A.}

The second part of the lemma follows by observing that the smooth  functions 
\[
t\mapsto  \bbom^\ve_{i,j|k,\ell}(t\be_1),\;\;t\mapsto \bbom^\ve_{-i,-j;k,\ell}(t\be_1),\;\;1\leq i,j,k,\ell\leq m,
\]
are even and,  as $\ve\to 0$, they  converge  in the topology of $C^\infty(\bR)$ to  the \emph{even} functions
\[
t\mapsto  \bbom^0_{i,j|k,\ell}(t\be_1),\;\;t\mapsto \bbom^0_{-i,-j|k,\ell}(t\be_1),\;\;1\leq i,j,k,\ell\leq m.
\]

\end{proof}
 From the above lemma, Corollary \ref{cor: silim} and (\ref{eq: vve_appr}) we deduce
\begin{corollary}
\begin{subequations}
\begin{equation}
\sup_{|t|\leq 1}t^{-2}\bigl|\,  \bbom^\ve_{\pm 1,j|k,\ell}(t\be_1)- \bbom^0_{\pm 1,j|k,\ell}(t\be_1)\,\bigr|= O(\ve^N),\;\;\forall j\geq 1,\;\; \ell\geq k\geq 1,
\label{eq: silim1a}
\end{equation}
\begin{equation}
\sup_{|t|\leq 1}\bigl|\,  \bbom^\ve_{\pm i,j|k,\ell}(t\be_1)- \bbom^0_{\pm i,j|k,\ell}(t\be_1)\,\bigr|= O(\ve^N),\;\;\forall 1<i\leq j,\;\;1<k\leq \ell.
\label{eq: silim1b}
\end{equation}
\end{subequations}
\label{cor: silim1}\qed
\end{corollary}

 \subsection{The behavior of $\rho_2^\ve$ in a neighborhood of the diagonal}\label{ss: 63}  Fix a point $\Theta_0$ on the diagonal $\bsD$. Without loss of generality we can assume that $\Theta_0=(0,0)\in(\bR^m/\bZ^m)\times (\bR^m/\bZ^m)$. Fix an open neighborhood $\eO_0$ of $(0,0)\in\bR^m\times\bR^m$ defined by 
\[
\eO_0=\bigl\{ (\vec{\theta},\vec{\vfi})\in\bR^m\times\bR^m;\;\;|\vec{\theta}\pm\vec{\vfi}|<\sqrt{2}\hbar\,\bigr\}.
\]
We regard $\eO_0$ as a neighborhood of $\Theta_0$ in $\bT^m\times \bT^m$.  Introduce a new system of orthogonal coordinates on $\eO_0$
\[
\vec{\omega}=\vec{\omega}(\Theta)= \frac{1}{\sqrt{2}}\bigl(\,\vec{\theta}+\vec{\vfi}\,\bigr),\;\;\vec{\nu}=\vec{\nu}(\Theta)= \frac{1}{\sqrt{2}}\bigl(\,\vec{\vfi}-\vec{\theta}\,\bigr).
\]
In these coordinates, the diagonal $\bsD\cap\eO_0$ is described  by  the equation
\[
\vec{\nu}=0.
\]
We have a natural projection
\[
\bpi:\eO_0\to \bsD \cap \eO_0,\;\; (\vec{\omega},\vec{\eta})\mapsto (\vec{\omega}, )\in\bsD\cap \eO_0.
\]
The projection $\bpi$ associates to a point $\Theta\in\eO_0$ the  (unique) point $\bpi(\Theta)\in\bsD\cap\eO_0$ clossest to $\Theta$.  The  vector $\vec{\nu}(\Theta)$    can be viewed as  a  vector in the  fiber  at $\bpi(\Theta)$  of the normal bundle $\eN_{\bsD}$.  We set
\[
\bsE_{\Upsilon^\ve\times \Upsilon^\ve}(|\det B|):=\int_{\Sym_m^{\times 2}}|\det B|\, d\Gamma_{\Upsilon^\ve\times \Upsilon^\ve}(|dB|),
\]
\[
\bsE_{\bbom^\ve(\tau^\ve(\Theta)\,)}(|\det B|):=\int_{\Sym_m^{\times 2}} |\det B| \, \Gamma_{\bbom^\ve(\tau^\ve(\Theta)\,)}(|dB|),
\]
where we recall that  the Gaussian measure  $\Gamma_{\Upsilon^\ve}$ is defined by (\ref{eq: upsivea}) and the Gaussian measure $\Gamma_{\bbom^\ve}$ is defined by (\ref{barbbom}).  In the coordinates $(\vec{\omega},\vec{\theta})$ we have $\tau(\Theta)=\sqrt{2}\vec{\nu}$.  Using (\ref{eq: tildro1}) and  (\ref{eq: rho2}) we deduce that on $\eO_0$ we have
\[
\rho^\ve_2(\Theta)-\tilde{\rho}_1^\ve(\Theta)
\]
\[
=\frac{\ve^{-2m}}{(2\pi)^m}\left(\underbrace{\frac{1}{\sqrt{\det\eH(\,V^\ve, \frac{\sqrt{2}}{\ve}\vec{\nu})}}\bsE_{\bbom^\ve(\,\frac{\sqrt{2}}{\ve}\vec{\nu}\,)}(|\det B|)}_{=:E_2^\ve(\frac{\sqrt{2}}{\ve}\vec{\nu})}\; -\; \underbrace{\frac{1}{\sqrt{\det\eH_\infty(V^\ve)}}\bsE_{\Upsilon^\ve\times \Upsilon^\ve}(|\det B|)}_{=:E_1^\ve}\right).
\]
The quantity $E_1^\ve$ is independed of  $\Theta$, while the quantity $E_2^\ve(\frac{\sqrt{2}}{\ve}\vec{\nu})$ depends only on the normal coordinate   $\vec{\nu}$.

\begin{lemma} There exists  a constant $C>0$ such that for any $\ve>0$ and any $|\vec{\eta}|$ sufficiently small we ahve
\begin{equation}
\bigl|\,E_2^\ve\bigl(\,\eta\,\bigr)\,\bigr|\leq C|\eta|^{2-m}
\label{eq: newt}
\end{equation}
\label{lemma: newt}
\end{lemma}

\begin{proof}  Using (\ref{eq: det}) we deduce that there exists a constant $C>0$ such that for any $\ve>0$  and $|\eta|$ sufficiently small
\begin{equation}
\frac{1}{C}|\eta|^{m}
\leq \sqrt{\det\eH^\ve(V^\ve, \eta)}\leq C|\eta|^{m},
\label{eq:  det-growth}
\end{equation}
Next, we need to estimate the  behavior of $\bsE_{\bbom^\ve(\,\eta\,)}(|\det B|)$ for $\eta$ small.

As explained at the end of Subsection \ref{ss: 53},  $\bbom^\ve(\eta)$ is  the covariance form  of  a of a Gaussian  measure on $\Sym^{\times 2}_m$. It describes a random  symmetric matrix  $B$ of the form
\[
B= B^+\oplus B^-,\;\;B^\pm\in\Sym_m,
\]
where the two components $B^\pm$ are  identically distributed. Their distribution is the Gaussian measure on $\Sym_m$  defined by the covariance form $(\bbom^\ve_{i,j|k,\ell})_{i,j,k,\ell\in\bsJ_m^+}$ detailed in (\ref{barbbom}). 

We now define a rescaling  $B^\eta$ of the random matrix $B$
\[
B^\eta=B^{-,\eta}\oplus B^{+,\eta},
\]
  where for $1\leq i\leq j\leq m$ we have
  \begin{equation}
  B^{\pm,\eta}_{ij}=\begin{cases}
  B^\pm_{ij}, & i>1,\\
  |\eta|^{-\frac{1}{2}}B^\pm_{1j}, &1=i<j,\\
|\eta|^{-1}B^\pm_{11}, & i=j=1.
  \end{cases}
  \label{eta_rescale}
  \end{equation}
 Observe that
 \begin{equation}
 \det B= |\eta|^2 \det B^\eta.
 \label{eq: resc_det}
 \end{equation}
 The rescaled matrix  $B^\eta$ is Gaussian,  with covariance form $\bbom^{\ve,\eta}=(\bbom^{\ve,\eta}_{i,j|k,\ell})_{i,j,k,\ell}$  described  by the equalities.
 \begin{equation}
 \begin{split}
 \bbom^{\ve,\eta}_{i,j|k,\ell}(\eta) = \bbom^{\ve,\eta}_{-i,-j|-k,-\ell}(\eta)=\bbom^\ve_{i,j|k,\ell}(\eta),\;\;\forall 1<i\leq j,\;\;1<k\leq \ell,\\
 \bbom^{\ve,\eta}_{1,j|k,\ell}(\eta) =\bbom^{\ve,\eta}_{-1,-j|-k,-\ell}(\eta)=|\eta|^{-\frac{1}{2}} \bbom^{\ve}_{1,j|k,\ell}(\eta),\;\;\forall j>1,\;\;\ell\geq k >1,\\
  \bbom^{\ve,\eta}_{1,j|1,\ell}(\eta) =\bbom^{\ve,\eta}_{-1,-j|-1,-\ell}(\eta)=|\eta|^{-1} \bbom^{\ve}_{1,j|1,\ell}(\eta),\;\;\forall j,\ell>1,\\
  \bbom^{\ve,\eta}_{1,1|1,\ell}(\eta) =\bbom^{\ve,\eta}_{-1,-1|-1,-\ell}(\eta)=|\eta|^{-\frac{3}{2}} \bbom^{\ve}_{1,1|1,\ell}(\eta),\;\;\forall \ell>1,\\
\bbom^{\ve,\eta}_{1,1|1,1}(\eta) =\bbom^{\ve,\eta}_{-1,-1|-1,-1}(\eta)=|\eta|^{-2} \bbom^{\ve}_{1,1|1,1}(\eta),
 \end{split}
 \label{eq: bbomet}
 \end{equation}
\begin{equation}
\begin{split}  \bbom^{\ve,\eta}_{-i,-j|k,\ell}(\eta) = \bbom^{\ve,\eta}_{i,j|-k,-\ell}(\eta)=\bbom^\ve_{-i,-j|k,\ell}(\eta),\;\;\forall 1<i\leq j,\;\;1<k\leq \ell,\\
 \bbom^{\ve,\eta}_{1,j|-k,-\ell}(\eta) =\bbom^{\ve,\eta}_{-1,-j|k,\ell}(\eta)=|\eta|^{-\frac{1}{2}} \bbom^{\ve}_{1,j|-k,-\ell}(\eta),\;\;\forall j>1,\;\;\ell\geq k >1,\\
  \bbom^{\ve,\eta}_{1,j|-1,-\ell}(\eta) =\bbom^{\ve,\eta}_{-1,-j|1,\ell}(\eta)=|\eta|^{-1} \bbom^{\ve}_{-1,-j|1,\ell}(\eta),\;\;\forall j,\ell>1,\\
  \bbom^{\ve,\eta}_{-1,-1|1,\ell}(\eta) =\bbom^{\ve,\eta}_{1,1|-1,-\ell}(\eta)=|\eta|^{-\frac{3}{2}} \bbom^{\ve}_{-1,-1|1,\ell}(\eta),\;\;\forall \ell>1,\\
\bbom^{\ve,\eta}_{-1,-1|1,1}(\eta) =\bbom^{\ve,\eta}_{1,1|-1,-1}(\eta)=|\eta|^{-2} \bbom^{\ve}_{-1,-1|1,1}(\eta).
 \end{split}
 \label{eq: bbometb}
 \end{equation}
 The above equalities coupled with Lemma \ref{lemma: O2} imply that  the limits
 \[
\bbom^{\ve,0}_{i,j|l,\ell}:= \lim_{|\eta|\to 0}\bbom^{\ve,\eta}_{i,j|k,\ell} (\eta)
 \]
 exist and are finite for any $i,j,k,\ell\in\bsJ_m$.  Thus the Gaussian measure $\Gamma_{\bbom^{\ve,\eta}(\eta)}$ converges as $|\eta|\to 0$ to   a  Gaussian measure\footnote{The limiting Gaussian measure  is degenerate, can be described explicitly, but  we will not need this  level of  detail.} $\Gamma_{\bbom^{\ve,0}}$.  Using (\ref{eq: resc_det})  we deduce
 \begin{equation}
 E_2^\ve(\eta)= E_{\bbom^{\ve}(\eta)}(|\det B|) =|\eta|^2 E_{\bbom^{\ve,\eta}}(\eta)(|\det B^\eta|)
 \label{eq: resc_cond_hess}
 \end{equation}
 so that
 \begin{equation}
 \lim_{|\eta|\to 0} |\eta|^{-2}E_{\bbom^{\ve}(\eta)}(|\det B|)= \bsE_{\bbom^{\ve,0}}(|\det C|)<\infty.
 \label{eq: exp-growth}
 \end{equation}
 The lemma now follows from the above equality coupled with the estimate (\ref{eq: det-growth}).
 \end{proof}

The estimate (\ref{eq:  newt}) shows that the function $E_2^\ve(\frac{\sqrt{2}}{\ve}\vec{\nu})$ is integrable on the tube $\eT_\hbar(\bsD)$.

\subsection{Proof of Theorem \ref{th: main}}\label{ss: 64} Using the notations in Section \ref{s: outline} we deduce
\[
\var^\ve=N_\ve+\int_{\bT^m\times \bT^m}\bigr(\rho_2^\ve(\Theta)-\tilde{\rho}_1^\ve(\Theta)\,\bigr)|d\Theta| 
\]
\[
=N_\ve+\int_{(\bT^M\times \bT^m)\setminus \eT_\hbar(\bsD)} \bigl(\,\rho^\ve_2(\Theta)-\tilde{\rho}_1^\ve(\Theta)\,\bigr)|d\Theta|+\int_{\eT_\hbar(\bsD)} \bigl(\,\rho^\ve_2(\Theta)-\tilde{\rho}_1^\ve(\Theta)\,\bigr)|d\Theta|.
\]
The estimate (\ref{eq: dif-ro}) implies 
\[
\int_{(\bT^M\times \bT^m)\setminus \eT_\hbar(\bsD)} \bigl(\,\rho^\ve_2(\Theta)-\tilde{\rho}_1^\ve(\Theta)\,\bigr)|d\Theta|=O(\ve^N),\;\;\forall N >0.
\]
Hence
\begin{equation}
\var^\ve=N_\ve+ \int_{\eT_\hbar(\bsD)} \bigl(\,\rho^\ve_2(\Theta)-\tilde{\rho}_1^\ve(\Theta)\,\bigr)|d\Theta|+O(\ve^N),\;\;\forall N>0.
\label{eq: varvediag}
\end{equation}
 For $\eta\in\bR^m$ we set
\[
\delta_\ve(\eta)= \begin{cases}
E_2^\ve(\eta)-E_1^\ve, & |\eta|\leq \frac{\hbar\sqrt{2}}{\ve}\\
0, &   |\eta| > \frac{\hbar\sqrt{2}}{\ve}.
\end{cases}
\]
Then
\[
\int_{\eT_\hbar(\bsD)} \bigl(\,\rho^\ve_2(\Theta)-\tilde{\rho}_1^\ve(\Theta)\,\bigr)|d\Theta|=\frac{\ve^{-2m}}{(2\pi)^m}{\rm vol}\,(\bsD)\int_{\bR^m} \delta_\ve\Bigl(\frac{\sqrt{2}}{\ve}\vec{\nu}\Bigr) |d\vec{\nu}|.
\]
If we make the change in variables $\eta=\frac{\sqrt{2}}{\ve}\vec{\nu}$,  and we observe that 
\[
{\rm vol}\,(\bsD)= 2^{\frac{m}{2}}\,{\rm vol}\,(\bT^m)= 2^{\frac{m}{2}},
\]
we deduce
\begin{equation}
\int_{\eT_\hbar(\bsD)} \bigl(\,\rho^\ve_2(\Theta)-\tilde{\rho}_1^\ve(\Theta)\,\bigr)|d\Theta|=\frac{\ve^{-m}}{(2\pi)^m}\int_{\bR^m} \delta_\ve(\eta) |d\eta|.
\label{eq: delve}
\end{equation}
\begin{lemma} For  any $\eta\in\bR^m\setminus 0$ the limit
\[
\delta_0(\eta)=\lim_{\ve\searrow 0}\delta_\ve(\eta)
\]
exists, the resulting function $\eta\mapsto \delta_0(\eta)$ is integrable on $\bR^m$ and  for any $N>0$
\begin{equation}
\lim_{\ve\searrow 0}\int_{\bR^m} \bigl(\,\delta_\ve(\eta)-\delta_0(\eta)\,\bigr)|d\eta|=O(\ve^N),\;\;\mbox{as $\ve\searrow 0$}.
\label{eq: del_lim}
\end{equation}
\label{lemma: punch}
\end{lemma}

\begin{proof} As in  Remark \ref{rem: upsi0}  we deduce that as $\ve\searrow 0$  the Gaussian measure  $\Gamma_{\Upsilon^\ve}$  converges to the Gaussian measure  $\Gamma_{\Upsilon^0}$ given by the covariance equalities (\ref{eq: upsi00}) and
\[
\lim_{\ve\searrow 0} \frac{1}{\sqrt{\det\eH_\infty(V^\ve)}}\bsE_{\Upsilon^\ve\times \Upsilon^\ve}(|\det B|)= \frac{1}{\sqrt{\det\eH_\infty(V)}}\bsE_{\Upsilon^0\times \Upsilon^0}(|\det B|).
\]
From Lemma \ref{lemma: xilim} we deduce that
\[
\lim_{\ve\searrow 0}\frac{1}{\sqrt{\det\eH(V^\ve, \eta)}}\bsE_{\bbom^\ve(\,\eta\,)}(|\det B|)=\frac{1}{\sqrt{\det\eH(V, \eta)}}\bsE_{\bbom^0(\,\eta\,)}(|\det B|).
\]
Hence
\begin{equation}
\delta_0(\eta)= \frac{1}{\sqrt{\det\eH(V, \eta)}}\bsE_{\bbom^0(\,\eta\,)}(|\det B|)-\frac{1}{\sqrt{\det\eH_\infty(V^)}}\bsE_{\Upsilon^0\times \Upsilon^0}(|\det B|).
\label{eq: delta0}
\end{equation}
Observe that
\[
\lim_{|\eta|\to \infty} \eH(V,\eta)=\eH_\infty(V),\;\;\lim_{|\eta|\to \infty}\bbom^0(\eta)=\Upsilon^0\times\Upsilon^0.
\]
Since $V$ is a Schwartz function   we deduce  that the functions
\[
\eta \mapsto  \eH(V,\eta)-\eH_\infty(V),\;\; \eta\mapsto  \bbom^0(\eta)-\Upsilon^0\times\Upsilon^0
\]
 have  fast decay at $\infty$,   i.e., faster than any power $\eta\mapsto |\eta|^{-N}$, $N>0$. Invoking Proposition \ref{prop: diff-gauss} we deduce that the function $\delta_0(\eta)$ also has  fast decay at $\infty$ and thus it is integrable at $\infty$.
 
We now argue as in the proof  of Lemma \ref{lemma: newt}.  Using the rescaling (\ref{eta_rescale}) and the equality  (\ref{eq: resc_det}) we deduce   that
 \begin{equation}
  E_{\bbom^{0}(\eta)}(|\det B|) =|\eta|^2 E_{\bbom^{0,\eta}(\eta)}(|\det B^\eta|),
 \label{eq: resc_cond_hess0}
 \end{equation}
  where $\bbom^{0,\eta}(\eta)$ is defined  by the equalities (\ref{eq: bbomet}) and (\ref{eq: bbometb}) in which $\ve$ is globally replaced by  the superscript $0$.  From the computations  in  Appendix \ref{s: b}  we deduce that  the Gaussian measures  $\bbom^{0,\eta}(\eta)$  have a limit as $\eta \to 0$. Using (\ref{det_explode}) we conclude
\[
\delta_0(\eta)=O(|\eta|^{2-m})
\]
 for $|\eta|$ small. This establishes the integrability of $\delta_0$  at the origin.

To prove  (\ref{eq: del_lim}) we will show that for any $N>0$
 \begin{subequations}
 \begin{equation}
\int_{|\eta|\geq 1}\bigl(\, \delta_\ve(\eta)- \delta_0(\eta)\,\bigr) |d\eta|=O(\ve^N),\;\;\mbox{as $\ve\searrow 0$}.
 \label{eq: del_lima}
 \end{equation}
 \begin{equation}
\int_{|\eta|\leq 1} \bigl(\,\delta_\ve(\eta)-\delta_0(\eta) \,\bigr)|d\eta|=O(\ve^N),\;\;\mbox{as $\ve\searrow 0$}.
 \label{eq: del_limb}
 \end{equation}
 \end{subequations}
 \noindent {\bf Proof of (\ref{eq: del_lima}).}  Observe that
 \[
 \int_{|\eta|\geq 1} \bigl(\,\delta_\ve(\eta)- \delta_0(\eta) \,\bigr)|d\eta| =\underbrace{\int_{1\leq |\eta|\leq \frac{\hbar\sqrt{2}}{\ve}}\bigr(\,\delta_\ve(\eta)-\delta_0(\eta)\,\bigr)|d\eta|}_{=:A_\ve} \; - \;\underbrace{\int_{|\eta|\geq \frac{\hbar\sqrt{2}}{\ve}}\delta_0(\eta)|d\eta|}_{=:B_\ve}.
 \]
 Since $\delta_0$ has fast decay at $\infty$ we deduce that
 \[
 B_\ve=O(\ve^N),\;\;\mbox{as $\ve\searrow 0$}\;\;\forall N>0.
 \]
 Observe that
 \[
 A_\ve=\int_{1\leq |\eta|\leq \frac{\hbar\sqrt{2}}{\ve}}\Biggl(\, \frac{1}{\sqrt{\det\eH(V^\ve, \eta)}}\bsE_{\bbom^\ve(\,\eta\,)}(|\det B|)-\frac{1}{\sqrt{\det\eH(V, \eta)}}\bsE_{\bbom^0(\,\eta\,)}(|\det B|)\,\Biggr)|d\eta|.
 \]
 We have $\det \eH(V,\eta) >0$ for any $|\eta|\neq 0$ and
 \[
 \lim_{|\eta|\to\infty}\eH(V,\eta)=\eH_\infty(V),\;\;\det \eH_\infty(V)\neq 0.
 \]
  Hence there exists $c>0$ such that
  \[
  \det \eH(V,\eta)> c,\;\;\forall |\eta|>1.
 \]
 Observe that   for any  $N,k>0$ we have
 \begin{equation}
 \sup_{|\eta|\leq \frac{\hbar\sqrt{2}}{\ve}} \bigl|\,\pa^\alpha V^\ve(\eta)-\pa^\alpha V(\eta)\,\bigr|=O(\ve^N),\;\;\forall|\alpha|\leq k,\;\;\mbox{as $\ve\searrow 0$}.
 \label{eq: vve0small}
 \end{equation}
 Hence
  \begin{equation}
 \sup_{|\eta|\leq \frac{\hbar\sqrt{2}}{\ve}}\bigl|\, \det \eH(V,\eta)-\det\eH_\infty(V)\,\bigr|=O(\ve^N),\;\;\mbox{as $\ve\searrow 0$}.
\label{eq: detvve0}
\end{equation}
The estimate (\ref{eq: vve0small}) implies   that for any $N>0$
\[
 \sup_{|\eta|\leq \frac{\hbar\sqrt{2}}{\ve}} \bigl|\,\bbom^\ve(\eta)-\bbom^0(\eta)\,\bigr|=O(\ve^N),\;\;\mbox{as $\ve\searrow 0$}.
 \]
Using Proposition \ref{prop: diff-gauss} we deduce that for any $N>0$
\begin{equation}
\sup_{|\eta|\leq \frac{\hbar\sqrt{2}}{\ve}} \Bigl|\, \bsE_{\bbom^\ve(\,\eta\,)}(|\det B|)-\bsE_{\bbom^0(\,\eta\,)}(|\det B|)\,\Bigr|=O(\ve^N),\;\;\mbox{as $\ve\searrow 0$}.
\label{eq: expve0}
\end{equation}
The estimates (\ref{eq: detvve0}) and (\ref{eq: expve0}) imply that $A_\ve=O(\ve^N)$, $\forall N>0$.

\medskip

\noindent {\bf Proof of (\ref{eq: del_limb}).}  We have
\[
\int_{|\eta|\leq 1} \bigl(\,\delta_\ve(\eta)-\delta_0(\eta) \,\bigr)|d\eta|
\]
\[
=\int_{ |\eta|\leq 1}\Biggl(\, \frac{1}{\sqrt{\det\eH(V^\ve, \eta)}}\bsE_{\bbom^\ve(\,\eta\,)}(|\det B|)-\frac{1}{\sqrt{\det\eH(V, \eta)}}\bsE_{\bbom^0(\,\eta\,)}(|\det B|)\,\Biggr)|d\eta|
\]
\[
=\int_{ |\eta|\leq 1}\Biggl(\, \frac{|\eta|^2}{\sqrt{\det\eH(V^\ve, \eta)}}\bsE_{\bbom^{\ve,\eta}(\,\eta\,)}(|\det B|)-\frac{|\eta|^2}{\sqrt{\det\eH(V, \eta)}}\bsE_{\bbom^{0,\eta}(\,\eta\,)}(|\det B|)\,\Biggr)|d\eta|.
\]
Lemma \ref{lemma: O2} implies that the functions
\[
\eta\mapsto \bbom^{\ve,\eta}(\eta),\;\;\ve\geq 0
\]
are continuous  on $|\eta|\leq 1$ and  Corollary \ref{cor: silim1} implies that
\[
\sup_{|\eta|\leq 1}\bigl| \,\bbom^{\ve,\eta}(\eta)-\bbom^{0,\eta}(\eta)\,\bigr|=O(\ve^N)\\;\;\forall N>0.
\]
 Invoking Proposition \ref{prop: diff-gauss} we  deduce
 \[
 \bigl|\, \bsE_{\bbom^{\ve,\eta}(\,\eta\,)}(|\det B|)-\bsE_{\bbom^{0,\eta}(\,\eta\,)}(|\det B|)\,\bigr|=O(\ve^N)\;\;\forall N>0.
 \]
 Using  (\ref{eq: det_appr}) we deduce that
 \[
 \sup_{|\eta|\leq 1}|\eta|^m \left|\frac{1}{\sqrt{\det\eH(V^\ve, \eta)}}-\frac{1}{\sqrt{\det\eH(V, \eta)}}\right|= O(\ve^N),\;\;\forall N>0.
 \]
 We conclude that
 \[
 \left|\int_{|\eta|\leq 1} \bigl(\,\delta_\ve(\eta)-\delta_0(\eta) \,\bigr)|d\eta|\,\right| = O\left(\int_{|\eta|\leq 1} \ve^N|\eta|^{2-m}|d\eta|\,\right),\;\;\forall N>0.
 \]
 This completes the proof of Lemma \ref{lemma: punch}. \end{proof}
 
 Using (\ref{eq: delve}) and Lemma \ref{lemma: punch} in (\ref{eq: varvediag})  we deduce
\begin{equation}
\var^\ve= N_\ve +\frac{\ve^{-m}}{(2\pi)^m}\left(\int_{\bR^m} \delta_0(\eta) |d\eta|\right)\cdot \bigl(1+O(\ve^N)\,\bigr),\;\;\forall N>0,
\label{eq: punch}
\end{equation}
where $\delta_0$ is described by (\ref{eq: delta0}).

  \appendix
  
\section{Some technical inequalities}
\label{s: a}
\setcounter{equation}{0}

\begin{proof}[Proof of Lemma \ref{lemma: tech}] Set 
\[
\be_1^\dag :=(1,0, 0,\dotsc, 0),\;\;\be_2^\dag:=(0,1,0,\dotsc,0)\in\bR^m.
\]
 For $t\in\bR$  we have
\[
f(t^2/2)=V( t\be_1^\dag).
\]
 Using (\ref{eq: d2}) we deduce that  for any $t\in\bR$ we have
 \[
 f'(t^2)+t f''(t^2/2)=\frac{d}{dt^2}  V(t\be_1^\dag)=\frac{d^2}{dt^2} \int_{\bR^m} e^{-\ii tx_1} w(|x|) |dx|=- \int_{\bR^m} e^{-\ii tx_1} x_1^2 w(|x|) |dx|.
 \]
Then 
\[
|f'(0)|=-f'(0) = \int_{\bR^m}  x_1^2 w(|x|) |dx|
\]
and
\[
|f'(0)|+  f'(t^2)+t f''(t^2/2)= \int_{\bR^m}  (1-e^{-\ii tx_1})x_1^2 w(|x|) |dx|
\]
\[
= 2 \int_{\bR^m} \bigl(\sin(tx_1/2)\,\bigr)^2 x_1^2 w(|x|) |dx|
\]
\[
|f'(0)|- \bigl(\, f'(t^2)+t f''(t^2/2)\,\bigr)= \int_{\bR^m}  (1+e^{-\ii tx_1})x_1^2 w(|x|) |dx|
\]
\[
= 2 \int_{\bR^m} \bigl(\, \cos(tx_1/2)\,\bigr)^2 x_1^2 w(|x|) |dx|.
\]
This proved (\ref{eq: technical}). To prove (\ref{eq: technicalb}) observe that
\[
f(t^2/2+s^2/2)=V(t\be_1^\dag+s\be_2^\dag)
\]
and we deduce that
\[
f'(t^2/2)=\pa^2_s f(t^2/2+s^2/2)|_{s=0} =\pa^2_s V(t\be_1^\dag+s\be_2^\dag)|_{s=0} =
\]
\[
\pa^2_s\left(\int_{\bR^m} e^{-\ii(tx_1+sx_2)}  w(x)|dx|\right)_{s=0}= -\int_{\bR^m} x_2^2  e^{-\ii tx_1} w(x)|dx|.
\]
We now conclude as before.\end{proof}

Let $\bsV$ be a real Euclidean space  of dimension $N$.  We denote by $\eA(\bsV)$ the space of symmetric positive semidefinite operators $A:\bsV\to\bsV$.  For $A\in \eA(\bsV)$ we denote  by $\bgamma_A$ the centered   Gaussian measure on $\bsV$ with  covariance form $A$. Thus
\[
\bgamma_{\one}(d\bv)=\frac{1}{(2\pi)^{\frac{N}{2}}}e^{-\frac{1}{2}|\bv|^2} d\bv,
\]
and $\bgamma_A$ is the push forward of $\bgamma_\one$ via the linear map $\sqrt{A}$,
\begin{equation}
\bgamma_A=(\sqrt{A})_* \bgamma_{\one}.
\label{eq: push}
\end{equation}
For any  measurable $f:\bsV\to\bR$ with at most polynomial growth  we set
\[
\bsE_A(f)=\int_{\bsV} f(\bv) \bgamma_A(d\bv).
\]
\begin{proposition} Let $f:\bsV\to \bR$ be a locally  Lipschitz function which is positively homogeneous of degree $\alpha\geq 1$.   Denote by $L_f$ the  Lipschitz constant of the restriction of  $f$ to the unit ball of $\bsV$.  There exists a constant $C>0$ which depends only on $N$ and $\alpha$  such that, for any $\Lambda>0$  and any $A,B\in \eA(\bsV)$ such that $\Vert A\Vert,\Vert B\Vert\leq \Lambda$ we have
\begin{equation}
\bigl|\, \bsE_A(f)-\bsE_B(f)\,\bigr|\leq  CL_f \Lambda^{\frac{\alpha-1}{2}} \Vert A- B\Vert^{\frac{1}{2}}.
\label{eq: holder}
\end{equation}
\label{prop: diff-gauss}
\end{proposition}

\begin{proof}    We present the very elegant argument we learned from George Lowther on \href{http://mathoverflow.net/questions/130496/continuous-dependence-of-the-expectation-of-a-r-v-on-the-probability-measure}{MathOverflow}. In the sequel we will use the same  letter $C$ to denote various constant that depend only on $\alpha$ and $N$.

First of all let us observe that  (\ref{eq: push}) implies that
\[
\bsE_A(f)=\int_{\bsV} f(\sqrt{A}\bv) \bgamma_{\one}(d\bv).
\]
We deduce that for any $t>0$ we have
\[
\bsE_{tA}(f) =\int_{\bsV} f(\sqrt{tA}\bv) \bgamma_{\one}(d\bv)=t^{\frac{\alpha}{2}}\int_{\bsV} f(\sqrt{A}\bv) \bgamma_{\one}(d\bv)=t^{\frac{\alpha}{2}}\bsE_A(f),
\]
and thus it suffices to prove (\ref{eq: holder}) in the special case $\Lambda=1$, i.e. $\Vert A\Vert,\Vert B\Vert\leq 1$.    We have
\[
\bigl|\, \bsE_A(f)-\bsE_B(f)\,\bigr|\leq \int_{\bsV} \bigl|\, f(\sqrt{A}\bv)-f(\sqrt{B}\bv)\,\bigr| \bgamma_{\one}(d\bv)
\]
\[
=\int_{\bsV}|\bv|^\alpha \Bigl|\, f\Bigl(\sqrt{A}\frac{1}{|\bv|}\bv\Bigr)-f\Bigl(\sqrt{B}\frac{1}{|\bv|}\bv\Bigr)\,\Bigr| \bgamma_{\one}(d\bv)
\]
\[
\leq L_f \int_{\bsV}|\bv|^\alpha \Bigl|\,\sqrt{A}\frac{1}{|\bv|}\bv-\sqrt{B}\frac{1}{|\bv|}\bv\,\Bigr|\bgamma_{\one}(d\bv)
\]
\[
\leq L_f\Vert\sqrt{A}-\sqrt{B}\Vert \int_{\bsV}|\bv|^\alpha \bgamma_{\one}(d\bv)\leq CL_f\Vert A-B\Vert^{\frac{1}{2}}.
\]

\end{proof}

\section {Limiting conditional hessians} 
\label{s: b}   

We include here  a more  explicit description of the covariances $\bbom^0_{i,j|k,\ell}(\eta)$.  Again we  fix an orthonormal basis $(\be_i)_{1\leq i\leq m}$, such that $\eta=|\eta|\be_1$. Then
\begin{equation}
\begin{split}
\bbom^0_{i,j|k,\ell}(\eta):= V_{i,j,k,\ell}(0) -\sum_{a,b>0}V_{ija}(\eta)V_{k,\ell, b}(\eta)\tsi^0_{a,b}(\eta),\\
 \bbom^0_{-i,-j|k,\ell}(\eta):=V_{i,j,k,\ell}(\eta) +\sum_{a,b>0}V_{i,j,a}(\eta)V_{k,\ell, b}(\eta)\tsi^0_{a,-b}(\eta),
\end{split}
\label{xi063}
\end{equation}
where, according to (\ref{eq: sinv}), we have
\[
\bigl(\, \tsi^0_{a,b}(\eta)\,\bigr)_{1\leq a,b\leq m} = -R(\eta) \bsH(V,0),\;\;\bigl(\, \tsi^0_{-a,b}(\eta)\,\bigr)_{1\leq a,b\leq m} = R(\eta) \bsH(V,\eta),
\]
and $R(\eta)$ is defined in (\ref{eq: rinv}). The equalities (\ref{eq: hesv}) and  (\ref{eq: sinv}) show that
\[
\tsi^0_{a,b}(\eta)=\tsi^0_{-a,b}(\eta)=0,\;\;\forall 1,\leq a,b\leq m, \;\;a\neq b.
\]
Set  as usual
\[
r:=\frac{1}{2}|\eta|^2.
\]
  The symmetric random matrix $A^0$   defined by the covariances   $\bbom^0_{\pm i, \pm j|k,\ell}(\eta)$  is a direct sum of two $m\times m$ random symmetric matrices $A^0=A^0_-\oplus A^0_+$. We    divide   a symmetric  $m\times m$ array  of numbers into four regions $a,b,c,d$ as depicted in the left hand side of Figure \ref{fig: 2}.   The array   defined by $A^0$  has  a corresponding partition depicted  in the right-hands side of Figure \ref{fig: 2}.
  
\begin{figure}[ht]
\centering{\includegraphics[height=2.5in,width=4in]{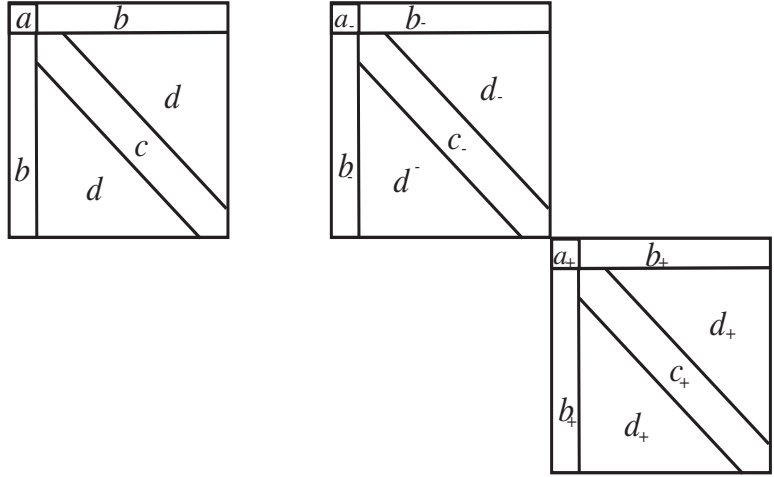}}
\caption{\sl Dividing a symmetric array into four parts.}
\label{fig: 2}
\end{figure}

If $u, v$   are two regions $u,v\in \{a_\pm,b_\pm,c_\pm,d_\pm\}$, then by  a $u-v$ correlation we mean  a correlation between   an entry of $A^0$ located in the region $u$ and an entry of $A^0$ located in the region $v$.

Using  (\ref{eq: rinv}) and (\ref{eq: sinv}) we deduce that   we have
\begin{equation}
\tsi_{1,1}^0(\eta)=\frac{-f'(0)}{f'(0)^2-f'(r)^2 -2|\eta|^2 f'(r)f''(r)-|\eta|^4f''(r)^2},\;\;V(\xi)=f(|\xi|^2/2).
\label{tsi11}
\end{equation}
The denominator of the above fraction admits a Taylor expansion ($r=\frac{|\eta|^2}{2}$)
\[
f'(0)^2-f'(r)^2 -2|\eta|^2 f'(r)f''(r)-|\eta|^4 f''(r)^2
\]
\[
=f'(0)^2 -\Biggl(\, f'(0) +\frac{f''(0)|\eta|^2}{2}  +\frac{f'''(0)|\eta|^4}{8}  \,\Biggr)^2
\]
\[
-2|\eta|^2\Bigl(f'(0)+\frac{f''(0)|\eta|^2}{2}\,\Bigr) \Bigl(f''(0)+\frac{f'''(0)|\eta|^2}{2}\Bigr)-|\eta|^4 f''(0)^2+O(|\eta|^6)
\]
\[
=-|\eta|^2\Biggl( 3f'(0)f''(0) + |\eta|^2\Bigl(\frac{9}{4}f''(0)^2+\frac{5}{4}f'(0)f'''(0)\,\Bigr)+O(|\eta|^4) \Biggr)
\]
Hence
\[
\tsi_{1,1}^0(\eta) =\frac{1}{3f''(0)|\eta|^2}\bigl( 1+c_{1,1}|\eta|^2+ O(|\eta|^4)\,\Bigr),\;\;c_{1,1}=-\frac{3}{4}\frac{f''(0)}{f'(0)}-\frac{5}{12}\frac{f'''(0)}{f''(0)}.
\]
Next,
\[
\tsi_{i,i}^0(\eta)=-\frac{f'(0)}{f'(0)^2-f'(r)^2}
\]
\[
=-\frac{f'(0)}{f'(0)^2 -\Bigl(\, f'(0) +\frac{f''(0)|\eta|^2}{2}  +\frac{f'''(0)|\eta|^4}{8}  \Bigr)^2+O(|\eta|^6)}.
\]
We conclude that 
\[
\tsi^0_{i,i}(\eta)= \frac{1}{f''(0)|\eta|^2}\frac{1}{ 1+\frac{1}{4}\Bigl(\frac{f'''(0)}{f''(0)}+\frac{f''(0)}{f'(0)}\Bigr)|\eta|^2+O(|\eta|^4)}
\]
\[
= \frac{1}{f''(0)|\eta|^2}\Bigl(1+c_0|\eta|^2+O(|\eta|^4)\,\Bigr),\;\; c_0=-\frac{1}{4}\Bigl(\frac{f'''(0)}{f''(0)}+\frac{f''(0)}{f'(0)}\Bigr),
\]
\[
\tsi_{-1,1}^0(\eta)=-\frac{f'(r)}{f'(0)} \tsi^0_{1,1}(\eta)
\]
\[
=-\frac{1}{3f''(0)|\eta|^2}\Bigl( 1+c_{1,1}|\eta|^2+ O(|\eta|^4)\,\Bigr)\Bigl( 1+\frac{f''(0)}{2f'(0)}|\eta|^2+O(|\eta|^4)\,\Bigr)
\]
\[
=-\frac{1}{3f''(0)|\eta|^2}\Bigl( 1+d_{1,1}|\eta|^2 +O(|\eta|^4)\,\Bigr),\;\;d_{1,1}=c_{1,1}+ \frac{f''(0)}{2f'(0)},
\]
\[
\tsi_{-i,i}^0(\eta)=-\frac{f'(r)}{f'(0)}\tsi^0_{i,i}
\]
\[
=-\frac{1}{f''(0)|\eta|^2} \Bigl( 1+d_{0}|\eta|^2 +O(|\eta|^4)\,\Bigr),\;\;d_{0}=c_{0}+ \frac{f''(0)}{2f'(0)}.
\]
\begin{center}
\ding{73}\ding{73}\ding{73}
\end{center}
 \[
 V_{1,1,1}(\eta)=3|\eta|f''(r)+|\eta|^3f'''(r)=|\eta|\Bigl(\, 3f''(0)+\frac{5}{2} |\eta|^2f'''(0) +O(|\eta|^4)\,\Bigr),
 \]
 \[
 V_{i,i,1}(\eta)=V_{i1i}(\eta)= |\eta|f''(r)=|\eta|\Bigl(\,f''(0)+\frac{1}{2}|\eta|^2 f'''(0) +O(|\eta|^4)\,\Bigr),
 \]
 \[
 V_{11i}(\eta)=0,\;\;i>1,
 \]
 \[
 V_{i,j,1}(\eta)=0,\;\;j>i>1,
 \]
 \[
 V_{i,j,k}(\eta)=0,\;\;i,j,k>1.
 \]
 
 \begin{center}
\ding{73}\ding{73}\ding{73}
\end{center}

\[
V_{i,j,k,\ell}(0)= \bigl(\, \delta_{ij}\delta_{k\ell}+ \delta_{ik}\delta_{j\ell} + \delta_{i\ell}\delta_{jk}\,\bigr)f''(0),
\]
\[
V_{1,1,1,1}(\eta) =3f''(r) + 6|\eta|^2 f'''(r) + |\eta|^4 f^{(4)}(r)=3f''(0)+\frac{15}{2}f'''(0)|\eta|^2+ O(|\eta|^4),
\]
\[
V_{1,1,1,i}(\eta)=0,\;\;i>1,
\]
\[
V_{1,1,i,i}(\eta)=V_{1,i,1,i}(\eta)=  f''(r)+|\eta|^2 f'''(r)=f''(0)+ \frac{3}{2}f'''(0)|\eta|^2+O(|\eta|^4),\;\;i>1,
\]
\[
V_{1,1,i,j}(\eta)=V_{1i1j}(\eta)=0,\;\;1<i<j,
\]
\[
V_{1,i,j,k}(\eta)=0,\;\;i,j,k>1.
\]
\begin{center}
\ding{86}\ding{86}
\end{center}
\[
V_{i,i,j,j}(\eta)=V_{ijij}(\eta)= f''(r),\;\;1<i<j,
\]
\[
V_{i,i,i,i}(\eta)=3f''(r),\;\;i>1,
\]
\[
V_{i,i,j,k}(\eta) = 0,\;\;  i>1,\; k>j>1,
\]
\[
V_{i,j,k,\ell}(\eta)=0,\;\; 1<i<j,\;\;1<k<\ell,\;\;(i,j)\neq (k,\ell). 
\]

\begin{center}
\ding{73}\ding{73}\ding{73}
\end{center}

If $1<i <j$, then
\[
 \bbom^0_{i,j|k,\ell}(\eta):= V_{i,j,k,\ell}(0) ,\;\;  \bbom^0_{-i,-j|k,\ell}(\eta):=V_{i,j,k,\ell}(\eta).
 \]
 
 \noindent $\bullet$  $d_+-d_+$   correlations. 
  \begin{subequations}
 \label{xiij}
 \begin{gather}
 \bbom^0_{i,j|k,\ell}(\eta)= 0,\;\;1<k\leq \ell,\;\;(i,j)\neq (k,\ell),\\
 \bbom^0_{i,j|i,j}(\eta) =V_{i,j,i,j}(0) -\sum_{a>0}V_{i,j,a}(\eta)V_{i,j, a}(\eta)\tsi^0_{a,a}(\eta)=V_{i,j,i,j}(0)=f''(0).
 \end{gather}
 \end{subequations}
 \noindent $\bullet$  $c_+-c_+$ correlations. If $1<i <j$, then
 \begin{subequations}
 \label{xcc}
 \begin{gather}
 \bbom^0_{i,i|j,j}(\eta)=V_{i,i,j,j}(0)-V_{1,i,i}(\eta)^2\tsi_{1,1}^0(\eta)= \frac{2}{3}f''(0) \Bigl(1+O(|\eta|^2)\Bigr),\\
 \bbom^0_{i,i|i,i}(\eta)=V_{i,i,i,i}(0)-V_{1,i,i}(\eta)^2\tsi_{1,1}^0(\eta)= \frac{8}{3}f''(0)\Bigl(1+O(|\eta|^2)\Bigr).
 \end{gather}
 \end{subequations}
 \noindent $\bullet$ The $a_+-a_+$, $a_+-b_+$, $a_+-c_+$ and $a_+-d_+$ correlations.  If $1<i <j$, then
 \begin{subequations}
\label{xi11}
 \begin{gather}
 \bbom^0_{1,1|1,1}(\eta)=  V_{1,1,1,1}(0)- V_{1,1,1}(\eta)^2\tsi^0_{1,1}(\eta) =\bar{c}_{1,1} |\eta|^2\bigl(1+O(\eta)^2\bigr),\\ \bar{c}_{1,1}=-9f'''(0)-3c_{1,1}f''(0),\\
  \bbom^0_{1,1|1,i}(\eta)=0,\;\;i>1,\\
 \bbom^0_{1,1|i,i}(\eta) =V_{1,1,i,i}(0)- V_{1,1,1}(\eta)V_{1,i,i}(\eta)\tsi^0_{1,1}(\eta) \sim const. |\eta|^2,\\
\bbom^0_{1,1| i,j}(\eta)=0.
\end{gather}
\end{subequations}
 \noindent $\bullet$  $b_+-b_+$ correlations. If $1<i <j$, then
\begin{subequations}
\label{b+b-}
\begin{gather}
 \bbom^0_{1,i|1,i}(\eta)= V_{1,i,1,i}(0)-V_{1,i,i}(\eta)^2\tsi^0_{i,i}(\eta)= \bar{c}_0|\eta|^2\Bigl(1+O(|\eta|^2)\Bigr),\\
 \bar{c}_0=-f'''(0)-f''(0)c_0,\\
 \bbom^0_{1,i|1,j}(\eta)= 0. \end{gather}
\end{subequations}
\noindent $\bullet$ The $b_+-c_+$ and $b_+-d_+$ correllations. If $1<i <j$, then
 \begin{subequations}
 \label{xi1i}
 \begin{gather}
 \bbom^0_{1,i|i,i}(\eta)=V_{1,i,i|i,i}(0)-\sum_{a} V_{1,i,a}(\eta) \tsi_{a,a}^0(\eta)V_{a,i,i}(\eta)=0,\\
  \bbom^0_{1,i| j,j}(\eta)=V_{1,i,j,j}(0)-\sum_{a} V_{1,i,a}(\eta) \tsi_{a,a}^0(\eta)V_{a,j,j}(\eta)=0,\\
   \bbom^0_{1,i|k,\ell }(\eta)= V_{i,j,k,\ell}(0)-\sum_{a} V_{1,i,a}\tsi_{a,a}^0(\eta)V_{a,k,\ell}(\eta)=0,\;\;\forall 1<k<\ell.
   \end{gather}
   \end{subequations}
  \noindent $\bullet$ The $a_--a_+$, $a_--b_+$, $a_--c_+$ and $a_--d_+$ correlations.  If $1<i <j$, then
 \begin{subequations}\label{xi-1}
 \begin{gather} 
 \bbom^0_{-1,-1|1,1}(\eta)=  V_{1,1,1,1}(\eta)+ V_{1,1,1}(\eta)^2\tsi^0_{-1,1}= \bar{d}_{1,1}|\eta|^2\Bigl( 1+O(|\eta|^2)\Bigr),\\
 \bar{d}_{1,1}=-3f''(0)d_{1,1}-\frac{3}{2}f'''(0),\\
 \bbom^0_{-1,-1|1,i}(\eta)=0,\;\;i>1,\\
 \bbom^0_{-1,-1| i,i}(\eta) =V_{1,1,i,i}(\eta)+V_{1,1,1}(\eta)V_{1,i,i}(\eta)\tsi^0_{-1,1}(\eta)=O( |\eta|^2),\\
 \bbom^0_{-1,-1| i,j}(\eta)= V_{1,1,i,j}(\eta) =0,\;\; 1<i<j.
 \end{gather}
 \end{subequations}
 \noindent $\bullet$  The $b_--b_+$, $b_--c_+$  and $b_--d_+$ correlations. 
 \begin{subequations}
 \label{b-d+}
 \begin{gather}
 \bbom^0_{-1,-i|1,i}(\eta)= V_{1,i,1,i}(\eta)+V_{1,i,i}(\eta)^2\tsi^0_{-i,i}(\eta)=\bar{d}_0|\eta|^2\Bigl( 1+O(|\eta|^2)\,\Bigr),\\
  \bar{d}_0=2f'''(0)-d_0f''(0),\;\;i>1,\\
 \bbom^0_{-1,-i| 1,j}(\eta)=0,\;\;1<i<j,\\
 \bbom^0_{-1,-i|j,j}(\eta)=0, \;\;i,j>1,\\
 \bbom^0_{-1,-i| j,k}(\eta) =V_{1,i,j,k}(\eta) =0,\;\;1<j<k,\;\;1<i.
 \end{gather}
 \end{subequations}
\noindent $\bullet$  The $c_--c_+$  and $c_--d_+$ correlations.
\begin{subequations}
\label{xi-i-j}
\begin{gather}
 \bbom^0_{-i,-i| j,j}(\eta)= V_{i,i,j,j}(\eta)+ V_{i,i,1}(\eta)^2\tsi^0_{-1,1}= \frac{2}{3}f''(0)\Bigl(1+O(|\eta|^2)\Bigr),\;\; 1<i<j, \\
 \bbom^0_{-i,-i|i,i}(\eta)=V_{i,i,i,i}(\eta)+ V_{i,i,1}(\eta)^2\tsi^0_{-1,1}(\eta)= \frac{8}{3}f''(0)\Bigl(1+O(|\eta|^2)\Bigr),\;\;i>1,\\
 \bbom^0_{-i,-i|j,k}=0,\;\;i>1, k>j>1.
 \end{gather}
 \end{subequations}
 \noindent $\bullet$ The $d_--d_+$ correlations.
 \begin{subequations}
 \begin{gather}
 \bbom^0_{-i,-j|i,j}(\eta)= V_{iijj}(\eta)=f''(r)=f''(0)\Bigl(1+O(|\eta|^2)\Bigr),\;\;1<i<j,\\
 \bbom^0_{-i,-j|k,\ell}(\eta)=0,\;\;1<i<j,\;\;1<k<\ell,\;\;(i,j)\neq (k,\ell).
 \end{gather}
 \end{subequations}
 
\section{Invariant random symmetric matrices}
\label{s: gmat}
\setcounter{equation}{0}

We denote by $\Sym_m$ the space of real symmetric   $m\times m$  matrices.   The group $O(m)$   acts by conjugation on $\Sym_m$. Would would like to describe  the collection $\eG_m$ of $O(m)$-invariant Gaussian measures on $\eS_n$.

Observe that $\Sym_m$ is an  Euclidean space with respect to the inner product
\begin{equation}
(A,B):=\tr(AB).
\label{nat-met}
\end{equation}
This inner product is invariant with respect to the action of $\SO(m)$ on $\Sym_m$. We set
\[
\widehat{\bsE}_{ij}:=\begin{cases}
\bsE_{ij}, & i=j\\
\frac{1}{\sqrt{2}}E_{ij}, & i<j.
\end{cases}.
\]
The collection  $(\widehat{\bsE}_{ij})_{i\leq j}$ is a basis  of $\Sym_m$ orthonormal with respect to the above inner product.  We  defined the \emph{normalized} entries
\begin{equation}
\hat{a}_{ij}:= \begin{cases}
a_{ij}, & i=j\\
\sqrt{2}a_{ij}, & i<j.
\end{cases}
\label{eq: norm-ent}
\end{equation}
The collection $(\hat{a}_{ij})_{i\leq j}$  the orthonormal basis of $\Sym_m\dual$ dual to $(\widehat{\bsE}_{ij})$.  The volume density induced by this metric is
\[
|dA|:=\Bigl\vert\prod_{i\leq j} d\widehat{a}_{ij}\,\Bigr\vert= 2^{\frac{1}{2}\binom{m}{2}}\Bigl\vert\,\prod_{i\leq j} da_{ij}\,\Bigr\vert.
\]
This volume density is $O(m)$-invariant.  Thus, the collection of $O(m)$-invariant Gaussian measures on $\Sym_m$  can be identified with the collection of positive definite $O(m)$-invariant quadratic forms  on $\Sym_m$.

The space of  $O(m)$-invariant  quadratic forms on $\Sym_m$ is   two dimensional and spanned  by the forms $\tr A^2$ and $(\tr A)^2$.  For any    real numbers  $u,v$ such that
\begin{equation}
v>0, mu+2v>0,
\label{eq: uv}
\end{equation}
 we  denote by  $d\bGamma_{u,v}(A)$ the centered Gaussian measure $d\bGamma_{u,v}(A)$ uniquely determined by the covariance equalities
 \[
 \bsE(a_{ij}a_{k\ell})=  u\delta_{ij}\delta_{k\ell}+ v(\delta_{ik}\delta_{j\ell}+ \delta_{i\ell}\delta_{jk}),\;\;\forall 1\leq i,j,.k,\ell\leq m.
 \]
 In particular we have
 \[
 \bsE(a_{ii}^2)= u+2v,\;\;\bsE(a_{ii}a_{jj})=u,\;\;\;\bsE(a_{ij}^2)=v,\;\;\forall 1\leq i\neq j\leq m,
 \]
while all other covariances are trivial.  We denote by $\Sym_m^{u,v}$ the space $\Sym_m$ equipped with the probability measure  $d\bGamma_{u,v}$  The  ensemble  $\Sym_m^{0,v}$ is  a rescaled version of of the Gaussian Orthogonal Ensemble  (GOE) and we will refer to it as $\GOE_m^v$.   

The  Gaussian measure $d\bGamma_{u,v}$ coincides with the Gaussian measure $d\bGamma_{u+2v,u,v}$ defined in \cite[App. B]{N2}.  We recall a few facts from   \cite[App. B]{N2}. 

The  probability density  $d\bGamma_{u,v}$  has the explicit description
\begin{equation}
d\bGamma_{u,v}(A)= \frac{1}{(2\pi)^{\frac{m(m+1)}{4}}  \sqrt{D(u,v)}} e^{-\frac{1}{4v}\tr A^2-\frac{u'}{2}(\tr A)^2} |dA|,
\label{eq: gamauv}
\end{equation}
 where
 \[
 D(u,v)= (2v)^{(m-1)+\binom{m}{2}}\bigl( mu +2v\,\bigr),
 \]
 and 
 \[
 u'=\frac{1}{m}\left(\frac{1}{mu+2v}-\frac{1}{2v}\right)=-\frac{u}{2v(mu+2v)}.
  \]
This shows that the  Gaussian measure  $d\Gamma_{u,v}$  is  $O(m)$-invariant.    Moreover the family $d\bGamma_{u,v}$, where $u,v$ satisfy (\ref{eq: uv}) exhausts $\eG_m$.

For $u>0$ the ensemble $\Sym_m^{u,v}$ can be given an alternate description.   More precisely   a random $A\in \Sym_m^{u,v}$ can be described as a sum
\[
A= B+ \ X\one_m,\;\;B\in \GOE_m^v,\;\; X\in \bsN(0, u),\;\;\mbox{ $B$ and $X$ independent}.
\]
We write  this
\begin{equation}
\Sym_m^{u,v} =\GOE_m^v\hat{+}\bsN(0,u)\one_m,
\label{eq: smgoe}
\end{equation}
where $\hat{+}$ indicates a sum of \emph{independent} variables.

In the special case  $\GOE_m^v$ we have $u=u'=0$  and  
\begin{equation}
d\bGamma_{0,v}(A)=\frac{1}{(2\pi v)^{\frac{m(m+1)}{4}} } e^{-\frac{1}{4v}\tr A^2} |dA|.
\label{eq: gov}
\end{equation}

Fix  a unit vector $\eta\in\bR^m$ and denote  by $O_\eta(m)$ the subgroup of  orthogonal  map $T:\bR^m\to\bR^m$ such that $T\eta=\eta$.    Group $O_\eta(m)$ continues to act by conjugation on $\Sym_m$ and we would like to describe the collection $\eG_{m,\eta}\supset \eG_m$ of $O_\eta(m)$-invariant Gaussian measures on $\Sym_m$.

The space of $O_\eta(m)$-invariant  quadratic forms on $\Sym_m$ is five dimensional and it is spanned by the quadratic forms\footnote{I learned this fact from \href{http://mathoverflow.net/questions/104751/invariants-of-symmetric-matrices}{Robert Bryant, on MathOverflow}.}
\[
\tr A^2,\;\;(\tr A)^2,\;\;(A^2\eta,\eta),\;\; (A\eta,\eta)^2,\;\;(\tr A)(A\eta,\eta).
\]
If we choose an orthonormal  basis $\be_1,\dotsc,\be_m$  of  $\bR^m$ such that $\eta=\be_1$,  then
\[
(A\eta,\eta)= a_{11},\;\; (A^2\eta,\eta)=\sum_{k=1}^m a_{1k}^2.
\]
If we block decompose a symmetric $m\times m$  matrix
\begin{equation}
A =\left[
\begin{array}{cc}
a_{11} & L^\dag\\
L & B
\end{array}
\right],
\label{eq: block}
\end{equation}
where $B\in\eS_{m-1}$, $b$ is a $(m-1)\times 1$ matrix, then we see that  a basis of the  space of  $O_\eta(m)$-invariant  is given by
\[
q_1(A)=(A\eta,\eta)^2 = \hat{a}_{11}^2,
\]
\[
 q_2(A)= 2|L|^2=\sum_{k=2}^m \hat{a}_{1k}^2=   2|A\eta|^2-2(A\eta,\eta)^2,
 \]
\[
q_3(A)=2\hat{a}_{11} \tr B=2 \hat{a}_{11} (\hat{a}_{22}+\cdots +\hat{a}_{mm})= 2(A\eta,\eta)\tr A-2(A\eta,\eta)^2,
\]
\[
 q_4(A)=(\tr B)^2 =(\hat{a}_{22}+\cdots +\hat{a}_{mm})^2=\bigl(\,\tr A- (A\eta,\eta),\bigr)^2
\]
\[
q_5(A)=\tr B^2= \tr A^2-2|L|^2 -a_{11}^2=\sum_{k=2}^m a_{kk}^2+ \sum_{2\leq i <j} \hat{a}_{ij}^2 .
\]
This suggests  dividing a  symmetric $m\times m$ array into four regions $a,b,c,d$ as in Figure \ref{fig: 1}

\begin{figure}[ht]
\centering{\includegraphics[height=2in,width=2in]{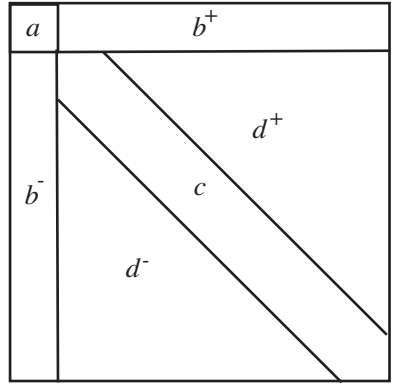}}
\caption{\sl Dividing a symmetric array into four parts.}
\label{fig: 1}
\end{figure}
More precisely, the parts  of these regions on or above the diagonal are
\[
a=\bigl\{\,(1,1)\,\bigr\},\;\; b^+= \bigl\{\, (1,j);\;\;j\geq 2\,\bigr\},\;\;c=\bigl\{\, (i,i);\;\;i\geq 2\,\bigr\},\;\; d^+=\bigl\{\, (i,j);\;\;i>j\geq 2\,\bigr\}.
\]
Suppose that $A$ is a Gaussian random symmetric  $m\times m$ matrix, where the Gaussian measure is $O_\eta(m)$-invariant. Then  the entries  withing the same region are identically distributed  Gaussian  variables with variance $v_a,\dotsc, v_d$.  Moreover, if  we fix regions $r_1,r_2\in \{a,b^+,c,d^+\}$, not necessarily distinct, and $x_1, x_2$ are   above or on the diagonal \emph{distinct normalized entries},  $x_1$ in the region $r_1$ and $x_2$ in the region $r_2$, then the covariance  $E(x_1,x_2)$ is independent of the location  of $x_1$ and $x_2$ in their respective regions, as long as $x_1\neq x_2$. We denote this covariance $\kappa_{r_1r_2}$.   Thus  a $O_\eta(m)$-invariant Gaussian measure on $\Sym_m$ is determined by a    variance vector
\[
\bv=(v_a,v_b,v_c,v_d),
\]
and a symmetric  covariance $4\times 4$ matrix
\[
K=(\kappa_{r_1r_2})_{r_1,r_2\in\{a,b,c,d\}}.
\]
When representing $K$ as a $4\times 4$ matrix we tacitly assume the order relation $a<b<c<d$.

The quantities $\bv$ and $K$  can be associated  to any $O_\eta(m)$-invariant quadratic form $q$, whether it is  positive semidefinite or not. We denote these quantities $\bv(q)$ and  $K(q)$.  Observe that
\[
\bv(xq+x'q')= x\bv(q)+x'\bv(q'),\;\;  K(xq+x'\bq')= xK(q)+x'K(q'),
\]
for any real numbers $x,x'$ and any  $O_\eta(m)$-invariant forms $q,q'$.  Observe that
\[
\bv(q_1)=(1,0,0,0),\;K(q_1)=0, 
\]
\[
\bv(q_2)=(0,1,0,0),\;\; K(q_2)=\left[
\begin{array}{cccc}
0 &  0& 0 & 0\\
0 & 0 & 0 &0 \\
0 & 0 & 0 & 0\\
0 & 0 & 0 & 0
\end{array}
\right],
\]
\[
\bv( q_3)=(0,0,0,0),\;\; K(q_3)=\left[
\begin{array}{cccc}
0 &  0& 1 & 0\\
0 & 0 & 0 & 0\\
1 & 0 & 0 & 0\\
0 & 0 & 0 & 0
\end{array}
\right],
\]
\[
\bv( q_4)=(0,0,1,0),\;\; K(q_4)=\left[
\begin{array}{cccc}
0 &  0& 0 & 0\\
0 & 0 & 0 &0 \\
0 & 0 & 1 & 0\\
0 & 0 & 0 & 0
\end{array}
\right],
\]
\[
 \bv( q_5)=(0,0,1,1),\;\; K(q_5)=0.
\]
Thus
\[
\bv( c_1q_1+c_2q_2+c_3q_3+c_4q_4+c_5 q_5)=(c_1, c_2,  c_5+c_4,c_5),
\]
\[
K( c_1q_1+c_2q_2+c_3q_3+c_4q_4+c_5 q_5)= \left[
\begin{array}{cccc}
0 &  0& c_3 & 0\\
0 & 0 & 0 & 0\\
c_3 & 0 &  c_4& 0\\
0 & 0 & 0 & 0
\end{array}
\right]
\]
We denote by $\Sym_m(c_1,\dotsc,c_5)$ the space $\Sym_m$ equipped with the $O_\eta(m)$-invariant  Gaussian measure with covariance 
\[
Q_{\vec{c}}:=c_1q_1+c_2q_2+c_3q_3+c_4q_4+c_5 q_5,
\]
whenever this quadratic form is positive semidefinite.  

For every  region $r\in \{a,b,c,d\}$ we denote by $\Sym_m(r)$ the vector subspace  of $\Sym_m$  consting of  matrices whose  entries  in regions other than $r$ are trivial.  We get an orthogonal direct sum decomposition
\[
\Sym_m=\Sym_m(a)\oplus \Sym_m(b)\oplus \Sym_m(c)\oplus \Sym_m(d).
\]
This leads to a corresponding decomposition
\[
A= A(a)+A(b)+ A(c) + A(d),  A\in\Sym_m.
\]
If  $A$  belongs to the ensemble $\Sym_m(c_1,\dotsc, c_5)$,  then

\begin{itemize}

\item   the components $A(a), A(b), A(c), A(d)$ are Gaussian  vectors, 

\item the component  $A(d)$ is independent of the rest of the other components, 

\item  the  $(m-1)\times (m-1)$-random matrix $A(c)+A(d)$ belongs to the ensemble $\Sym_{m-1}^{c_4,c_5}$.

\end{itemize}

To decide when $Q_{\vec{c}}\geq 0$ we need to  compute the   eigenvalues of the  the matrix  $\eQ_{\vec{c}}$  describing $Q_{\vec{c}}$  in  the  orthonormal coordinates $\hat{a}_{ij}$, $1\leq i\leq j\leq m$.    

For any $r\in\{a,b,c,d\}$ denote by $Q(r)$ the matrix   representing   the restriction of $Q_{\vec{c}}$ to $\Sym_m(r)$.   Moreover, for  any  positive integer $n$, we denote by $C_n$ the symmetric $n\times n$ matrix all whose  entries are equal to $1$.We have
\[
Q(a)= c_1,\;\;Q(d)=c_5 \one_N,\;\; N=\dim\Sym_{m-1}(d)=\frac{(m-1)(m-2)}{2},
\]
\[
Q(b)=c_2\one_{m-1},\;\; Q(c)=c_5\one_{m-1}+c_4 C_{m-1}
\]
Finally, denote by $L_n$ the $1\times n$ matrix whose entries are all equal to $1$.  With this notation we deduce that   $\eQ_{\vec{c}}$  has the block decomposition 
\[
\eQ_{\vec{c}}=\left[
\begin{array}{cccc}
Q(a) & c_2 L_{m-1} & c_3 L_{m-1} & 0\\
c_2 L_{m-1}^\dag& Q(b) & 0 & 0\\
c_3L_{m-1}^\dag & 0 & Q(c) & 0\\
0 & 0 & 0  & Q(d)
\end{array}
\right]=\left[
\begin{array}{cccc}
c_1 & 0 & c_3 L_{m-1} & 0\\
0&  c_2\one_{m-1} & 0 & 0\\
c_3L_{m-1}^\dag & 0 & c_5\one_{m-1}+c_4 C_{m-1}& 0\\
0 & 0 & 0  & c_5 \one_N\end{array}
\right].
\]
To compute the spectrum  of  $\eQ_{\vec{c}}$ it suffices to compute the spectrum of the matrix
\[
 \left[
\begin{array}{ccc}
c_1 & 0& c_3 L_{m-1} \\
0&  c_2\one_{m-1} & 0 \\
c_3L_{m-1}^\dag & 0 & c_5\one_{m-1}+c_4 C_{m-1}\end{array}
\right],
\]
viewed as a symmetric operator acting on the  space $\bR\oplus \bR^{m-1}\oplus\bR^{m-1}$. We see that  the subspace $0\oplus \bR^{m-1}\oplus 0$  is an invariant subspace of this operator  and its restriction to this subspace is $c_2\one_{m-1}$. Thus it suffices to find the  spectrum of the operator
\[
\bar{\eQ}_{\vec{c}}= \left[
\begin{array}{cc}
c_1 &  c_3 L_{m-1} \\

c_3L_{m-1}^\dag  & c_5\one_{m-1}+c_4 C_{m-1}\end{array}
\right]
\]
acting  on $\bsV=\bR\oplus \bR^{m-1}$.  A vector $\bv\in\bsV$ has a decomposition
\[
\bv=\left[
\begin{array}{c}
t\\
\bx
\end{array}
\right],\;\;\bx=\left[
\begin{array}{c}
x_1\\
\vdots\\
x_{m-1}
\end{array}\right].
\]
We set
\[
\si(\bx)=x_1+\cdots+x_{m-1},\;\;\bu=L_{m-1}^\dag=\left[
\begin{array}{c}
1\\
\vdots\\
1
\end{array}\right]\in\bR^{m-1}.
\]
We have
\begin{equation}
\bar{\eQ}_{\vec{c}} \left[
\begin{array}{c}
t\\
\bx
\end{array}
\right]=\left[
\begin{array}{c}
c_1 t+c_3\si(\bx)\\
c_3t\bu+ c_5\bx+ c_4\si(\bx)\bu.
\end{array}
\right].
\label{eq: large-mat}
\end{equation}
This  shows that the codimension $2$ subspace $\bsV_0\subset \bsV$  given by
\[
t=0,\;\;\si(\bx)=0,
\]
is $\bar{\eQ}_{\vec{c}}$-invariant, and the restriction of $\bar{\eQ}_{\vec{c}}$ to this subspace  is $c_5\one_{\bsV_0}$.   

An orthogonal  basis   of the orthogonal complement  $\bsW_0:=\bsV_0^\perp$ is given by the vectors
\[
\bv_0=\left[
\begin{array}{c}
1 \\
0
\end{array}
\right] ,\;\;\bv_1=\left[
\begin{array}{c}
0 \\
\bu
\end{array}
\right].
\]
Using (\ref{eq: large-mat}) we deduce
\[
\bar{\eQ}_{\vec{c}}(t_0\bv_0+t_1\bv_1)=\bigl(\, c_1t_0+(m-1)c_3t_1\,\bigr)\bv_0
\]
\[
+\bigr(\,c_2t_0+c_5t_1\bigl)\bv_1 +\bigl(\,c_3t_0+(c_5+(m-1)c_4) t_2\,\bigr)\bv_2.
\]
Thus in the basis $\bv_0,\bv_1,\bv_2$ of $V_0^\perp$ the  restriction of $\bar{\eQ}_{\vec{c}}$ to this subspace is given by the  $2\times 2$ matrix\footnote{The matrix $\Sigma_{\vec{c}}$ is not symmetric since the basis $\bv_0,\bv_1,\bv_2$ is not \emph{orthonormal}.}
\[
\Sigma_{\vec{c}}=\left[
\begin{array}{ccc}
c_1 &  (m-1)c_3\\
c_3 &  c_5 +(m-1)c_4
\end{array}
\right].
\]
Putting together all the above facts we obtain the following result.
\begin{proposition} The  quadratic form $Q_{\vec{c}}$ is positive definite if and only if  $c_1, c_2,c_5>0$  and  $\det\Sigma_{\vec{c}}>0$. Moreover
\begin{equation}
\det Q_{\vec{c}}= \Delta(\vec{c})= c_2^{m-1}c_5^{N+\dim \bsV_0}\det\Sigma_{\vec{c}}= c_2^{m-1}c_5^{\frac{(m-1)(m-2)}{2}+m-2} \det\Sigma_{\vec{c}}.
\label{eq: qc}
\end{equation}\qed
\label{prop: qc}
\end{proposition}

\begin{remark}\label{rem: gmat} (a) The  matrix  $\Sigma_{\vec{c}}$ is similar to the symmetric matrix
\[
\hat{\Sigma}_{\vec{c}}=\left[
\begin{array}{ccc}
c_1 &  (m-1)^{\frac{1}{2}}c_3\\

(m-1)^{\frac{1}{2}}c_3 &  c_5 +(m-1)c_4
\end{array}
\right].
\]
Thus $\eQ_{\vec{c}}$ is positive definite iff $c_2, c_5>0$ and the symmetric matrix $\hat{\Sigma}_{\vec{c}}$ is positive definite.

\noindent (b)   The above proof  produced an orthogonal decomposition  of $\Sym_m$
\[
\Sym_m= \Sym_m(d)\oplus \Sym_m(b) \oplus \bsV,\;\;\bsV=\Sym_m(a)\oplus \Sym_m(c).
\]
In the proof  we used the  natural metric (\ref{nat-met}) on $\Sym_m$ to identify $Q_{\vec{c}}$ with a symmetric operator $\Sym_m\to\Sym_m$. The three factors above  are invariant  subspaces of this operator.  The restriction of $Q_{\vec{c}}$ to $\Sym_m(d)$ is $c_5\one$, while the restriction of $Q_{\vec{c}}$ to $\Sym_m(b)$ is $c_2\one$.  The subspace $\bsV$  decomposes further into two invariant subspaces:  the two-dimensional  subspace
\[
\bsW_0= \Sym_m(a)\oplus \spa\one_{m-1}=\bR\oplus \spa{\one}_{m-1},\;\;\one_{m-1}\in\Sym_m(c),
\]
and its orthogonal complement $\bsV_0$. The restriction  of $Q_{\vec{c}}$ to $\bsV_0$ is $c_5\one_{\bsV_0}$ while  the restriction to  $\bsW_0$ is  described in the canonical orthonormal basis
\[
1\oplus 0,\;\;0\oplus\frac{1}{\sqrt{m-1}}\one_{m-1}
\]
by the matrix  $\hat{\Sigma}_{\vec{c}}$.

\noindent (c) We denote by $\eT_\eta$ the space of $O_\eta(m)$-equivariant symmetric operators 
\[
T:\Sym_m\to \Sym_m
\]
The above discussion  shows that any $T\in\eT_\eta$ enjoys the following properties.

 \begin{itemize}
 
 \item  Each of the subspaces $\Sym_m(d)$, $\Sym_m(b)$, $\bsV_0$ and $\bsW_0$  are $T$-invariant, where $\bsV_0$ and $\bsW_0$ are defined as in (b). 
 
 \item There exist two real constants $\alpha,\beta$ such that the restriction of $T$ to $\Sym_m(b)$ is $\alpha\one$, while the restriction of $T$ to $\Sym_m(d)\oplus \bsV_0$ is $\beta\one$. ($\alpha=c_2$, $\beta=c_5$.)
 
 \end{itemize}
 
 We denote by $\hat{\Sigma}_T$ the restriction of $T$ to $\bsW_0$. We see that an operator $T\in\eT_\eta$ is determined by a triplet $(\alpha,\beta, S)$ where $\alpha,\beta\in\bR$ and $S:\bsW_0\to\bsW_0$ is a symmetric operator. We  will denote by $T_{\alpha,\beta, S}$ the operator associated to this triplet. Note also that $T_{\alpha,\beta,S}$ is invertible iff $\alpha\beta\det S\neq 0$ and
\[
T_{\alpha,\beta,S}^{-1}= T_{\alpha^{-1},\beta^{-1},S^{-1}}.
\]
A matrix $A\in \Sym_m$ has the  block form  (\ref{eq: block})
\[
A=\left[
\begin{array}{cc}
a_{11} & L^\dag\\
L & B
\end{array}
\right]
\]
we further decompose $B$ as a sum
\[
B:=\frac{c}{\sqrt{m-1}}\one_{m-1}+ B_0,\;\;\tr B_0 =0,
\]
so that $c\sqrt{m-1}=\tr B$.  We  we will refer to this matrix using the notation $A=A(a_{11}, L, c, B_0)$.  Then
\[
T_{a,b,S} A(a_{11}, L,c ,B_0)= A(a_{11}', L',c' ,B_0'),
\]
where
\[
B_0'=\beta B_0,\;\;L'= \alpha L,\;\;\left[
\begin{array}{cc}
a_{11}' \\
c'
\end{array}
\right]= S \left[
\begin{array}{cc}
a_{11} \\
c
\end{array}
\right].\proofend
\]
\end{remark}

\end{document}